\documentclass[a4paper]{amsart}
\usepackage{amsmath,amssymb,amsthm,enumerate,graphicx}

\title[Hyperk\"ahler ambient metrics]
{Hyperk\"ahler ambient metrics associated with twistor CR manifolds}

\author{TAIJI MARUGAME}
\date{}

\renewcommand\a{\alpha}
\renewcommand\b{\beta}
\newcommand\g{\gamma}
\renewcommand\d{\delta}
\newcommand\e{\varepsilon}

\renewcommand\th{\theta}

\newcommand\bo{\boldsymbol{o}}
\newcommand\U{\Upsilon}

\newcommand{\wt}{\widetilde}
\newcommand{\wh}{\widehat}

\newtheorem{lem}{Lemma}[section]

\newtheorem{theorem}[lem]{Theorem}
\newtheorem{prop}[lem]{Proposition}

\theoremstyle{definition}

\newtheorem{rem}[lem]{\it Remark}

\numberwithin{equation}{section}

\makeatletter
\@addtoreset{equation}{section}

\makeatother
\address{Department of Mathematics, The University of Electro-Communications, 1-5-1 Chofugaoka, Chofu, Tokyo 182-8585, Japan}
\email{marugame@uec.ac.jp}



\begin{document}

\begin{abstract} 
Twistor CR manifolds, introduced by LeBrun, are Lorentzian (neutral) CR 5-manifolds defined as  $\mathbb{P}^1$-bundles over 3-dimensional conformal manifolds. 
In this paper, we embed a real analytic twistor CR manifold into the twistor space of the anti self-dual Poincar\'e-Einstein metric whose conformal infinity is the base conformal 3-manifold, and construct the associated Fefferman ambient metric as a neutral hyperk\"ahler metric on the spinor bundle with the zero section removed.
We also describe the structure of the Cheng--Yau type K\"ahler-Einstein metric which has the twistor CR manifold as the boundary at infinity. 
\end{abstract}
\maketitle


\section{Introduction}
The Fefferman--Graham ambient metric (\cite{FG1, FG2}) is an (asymptotically) Ricci-flat pseudo-Riemannian metric $\wt g$ of signature $(p+1, q+1)$ associated with any conformal manifold $(M, [g])$ of signature $(p, q)$. It is defined on a manifold $\wt{\mathcal{G}}$, callled the ambient space, which is diffeomorphic to the product $\mathcal{G}\times(-\e, \e)_\rho$ of the $\mathbb{R}_+$-bundle $\mathcal{G}\subset S^2(T^*M)$ determined by $[g]$ and a small interval $(-\e, \e)$ with the coordinate $\rho$. The Ricci curvature of $\wt g$ satisfies
\[
{\rm Ric}(\widetilde g)=\begin{cases}
O(\rho^\infty) & (n: {\rm odd}) \\
O^+(\rho^{n/2-1}) & (n:  {\rm even}),
\end{cases}
\]
where $n$ is the dimension of $M$; see \S\ref{ambient}  for the notation $O^+(\rho^m)$. This equation, together with a homogeneity condition and a boundary condition, uniquely determines the jet of $\wt g$ along $\mathcal{G}\times\{0\}$ to infinite order when $n$ is odd  and to a finite order when $n$ is even. Moreover, when $(M, [g])$ is real analytic and $n$ is odd, $\wt g$ converges to a Ricci-flat metric in a neighborhood of $\mathcal{G}\times\{0\}$.
Since it is canonically attached to conformal manifolds, the ambient metric provides a powerful tool to construct conformal invariants and conforamlly invariant differential operators (see, e.g., \cite{FG1, FG2, FH, GJMS}).

There is also a notion of ambient metric in CR geometry, which is the geometry of real hypersurfaces in complex manifolds.
Although the CR ambient metric can be regarded as a special case of the Fefferman--Graham ambient metric, it was first introduced by Fefferman in his pioneering work \cite{F}, prior to the construction of the ambient metric in conformal geometry. Let $M$ be a Levi-non-degenerate real hypersurface in an $(n+1)$-dimensional complex manifold $Y$ with a holomorphic line bundle $\mathcal{L}\to Y$. Then, the ambient metric $\wt g$ for $M$ is defined to be a Ricci-flat pseudo-K\"ahler metric on $\mathcal{L}\setminus\{\bo\}$ whose K\"ahler form is written as 
$i\partial\overline{\partial} \wt r$ with a homogeneous defining function $\wt r$ of $(\mathcal{L}\setminus\{\bo\})|_M$. The Ricci-flat condition is equivalent to the complex Monge--Amp\`ere equation for $\wt r$, and Fefferman \cite{F} constructed an approximate solution to it in the case where $Y=\mathbb{C}^{n+1}$ and $\mathcal{L}=K^{1/(n+2)}_{\mathbb{C}^{n+1}}$.
When the Levi form of $M$ has signature $(p, q)$, the restriction of $\wt g$ to $(\mathcal{L}\setminus\{\bo\})|_M$ defines a conformal structure of signature $(2p+1, 2q+1)$, called the Fefferman metric, on the $S^1$-bundle $(\mathcal{L}\setminus\{\bo\})|_M/\mathbb{R}_+$ over $M$. Thus, $\wt g$ can be reinterpreted as the Fefferman--Graham ambient metric for this even dimensional conformal structure. 
There is also an intrinsic definition of the Fefferman metric for abstract CR manifolds; see \cite{BDS, Lee}. 

The K\"ahlerity of the ambient metric in the CR case reflects the fact that the Fefferman metric is a special conformal structure induced by a CR structure.
We recall that the flat model of conformal manifold of signature $(2p+1, 2q+1)$ is $S^{2p+1}\times S^{2q+1}$ and its symmetry group is isomorphic to $\mathrm{O}(2p+2, 2q+2)$ (up to a finite covering), which agrees with the isometry group of the flat ambient space $\mathbb{R}^{2p+2, 2q+2}$. On the other hand, the flat model of CR manifold of signature $(p, q)$ is the hyperquadric in $\mathbb{CP}^{n+1}$ and its symmetry group is 
$\mathrm{SU}(p+1, q+1)$ (up to a finite covering), which is a proper subgroup of $\mathrm{O}(2p+2, 2q+2)$. The K\"ahlerity of $\wt g$ means that one can construct an ambient metric for  the Fefferman metric so that its holonomy group is contained in this smaller group. This is an example of the general principle that the ambient metric for a special class of conformal manifolds obtained from another geometric structure may have a special holonomy; see \cite{CGGH} for the general correspondence in the infinitesimal level between the holonomy of the ambient metric and that of the normal tractor connection associated to conformal manifolds.

There is another example of a special ambient metric.
Following the work of Nurowski \cite{N1, N2} and Leistner--Nurowski \cite{LN},  
Graham--Willse \cite{GW} proved that the conformal structure on a 5-manifold induced by a generic 2-plane distribution admits an ambient metric whose holonomy group is contained in the exceptional Lie group $\mathrm{G}_2\subset \mathrm{SO}(3, 4)$.
For the proof, they established the ``parallel tractor extension'' theorem.
A conformal manifold has a natural vector bundle, called the tractor bundle, equipped with a conformally invariant connection called the tractor connection. The tractor bundle is isomorphic to $T\wt{\mathcal{G}}|_{\mathcal{G}}$ divided by an $\mathbb{R}_+$-action and the 
tractor connection coincides with the restriction of the Levi-Civita connection $\wt\nabla$ of the ambient metric $\wt g$. A section of the tractor bundle or its various tensor bundles is called a tractor. In \cite{GW}, it is shown that any parallel tractor $\chi$ admits an extension to a tensor field $\wt\chi$ on $\wt{\mathcal{G}}$ satisfying
\[
\widetilde\nabla\widetilde\chi=
\begin{cases}
O(\rho^\infty) & (n: {\rm odd}) \\
O(\rho^{n/2-1}) & (n:  {\rm even}).
\end{cases}
\]
Moreover, if $n$ is odd and the conformal manifold is real analytic, 
 $\wt\chi$ satisfies $\widetilde\nabla\widetilde\chi=0$ in a neighborhood of $\mathcal{G}$. The condition that the holonomy group of $\wt g$ be contained in $\mathrm{G}_2$ is characterized by the existence of the parallel associated 3-form, and the existence of the corresponding parallel tractor had been shown by Hammerl--Sagerschnig \cite{HS}. By applying the parallel tractor extension theorem to this tractor, Graham--Willse proved that $\wt g$ has $\mathrm{G}_2$-holonomy in the real analytic case. 
 
The aim of this paper is to present yet another example of ambient metric with a special holonomy; we construct ambient metrics associated with a class of CR manifolds, called twistor CR manifolds, as neutral hyperk\"ahler metrics in the real analytic case. Twistor CR manifolds, introduced by LeBrun \cite{LeB2}, are Lorentzian (neutral) CR 5-manifolds defined as  $\mathbb{P}^1$-bundles over (positive definite) conformal 3-manifolds.
When the base conformal manifold is the standard sphere $S^3$, the twistor CR manifold is CR isomorphic to the hyperquadric
\[
\mathcal{N}=\bigl\{ [(W^\a)]\in \mathbb{P}^3\ \big|\ 
W^0\overline{W^2}+W^2\overline{W^0}+W^1\overline{W^3}+W^3\overline{W^1}=0\bigr\},
\]
which appears in Penrose's twistor theory as the space of  light rays (null twistors) in the Minkowski spacetime. In this case, the hyperk\"ahler ambient metric coincides, up to a constant multiple, with the flat neutral metric
\[
\wt g=dW^0\cdot d\overline{W^2}+dW^2\cdot d\overline{W^0}+dW^1\cdot d\overline{W^3}+dW^3\cdot d\overline{W^1}
\]
on $\mathbb{C}^{4}$; we prove this for the case of $\mathbb{R}^3$ instead of $S^3$ in \S\ref{flat-case}.

The Fefferman metric of a twistor CR manifold is a conformal structure of neutral signature on a 6-dimensional manifold, but it emerges from 3-dimensional conformal geometry, whose symmetry group is $\mathrm{O}(4, 1)$. We recall that there is an isomorphism
\[
\mathrm{Spin}(4, 1)\cong \mathrm{Sp}(1, 1) \bigl(\subset \mathrm{SU}(2, 2)\bigr),
\]
which corresponds to the equivalence $B_2=C_2$ of the Dynkin diagrams. This explains why we can expect the existence of hyperk\"ahler ambient metrics for twistor CR manifolds. However, since the Fefferman conformal structure is even dimensional, the Ricci-flat equation is not guaranteed to be solved to infinite order in general, and also Graham--Willse's parallel tractor extension theorem can only be applied up to a finite order. Thus, we cannot reduce the construction of hyperk\"ahler ambient metrics to that of parallel tractors by using the general theory.

To construct hyperk\"ahler ambient metrics, we first embed the twistor CR manifold into a 3-dimensional complex manifold given as the twistor space of an anti self-dual conformal 4-manifold. It is proved by LeBrun \cite{LeB2} that a twistor CR manifold has a remarkable property that it never admits a local CR embedding into a complex manifold unless the base conformal 3-manifold $(\Sigma, [h])$ is real analytic.
When $(\Sigma, [h])$ is real analytic, another result of LeBrun \cite{LeB1} shows that $(\Sigma, [h])$ can be realized as the two-sided conformal infinity of an anti self-dual Poincar\'e-Einstein manifold $(X\setminus\Sigma, g_+)$. We assume that $\Sigma$ is orientable. Then, since $\Sigma$ is 3-dimensional, its tangent bundle is trivial. Shrinking $X$ if necessary, we may also assume that $X$ has the trivial tangent bundle, so we fix a spin structure of the compactified conformal manifold $(X, [\overline{g}_+])$.
By the construction of Atiyah--Hitchin--Singer \cite{AHS}, we have the twistor space of $(X, [\overline{g}_+])$ as the projectivization 
$\mathbb{P}(\mathbb{S}')$ of the spinor bundle $\mathbb{S}'\to X$. We show that 
the twistor CR manifold is canonically isomorphic to the real hypersuface $\mathbb{P}(\mathbb{S}')|_\Sigma$ and construct a hyperk\"ahler ambient metric on the total space $\mathbb{S}'\setminus\{\bo\}$.

In fact, the construction of the hyperk\"ahler metric works in a more general setting; 
there exists a hyperk\"ahler metric on the spinor bundle with the zero section removed over any 4-dimensional anti self-dual Einstein manifold having the non-zero scalar curvature. Thus, we first present the general construction in \S\ref{hyperkahler-spinor} before we discuss the ambient metric. The existence of such metrics, however, has already been proved by Swann \cite{Sw}:
\begin{theorem}[{\cite[Corollary 3.6]{Sw}}]\label{swann}
Let $(X, g)$ be a $4$-dimensional anti self-dual Einstein manifold with scalar curvature $R\neq0$, and $\mathbb{S}'$ the spinor bundle over $X$.
Then, $\mathbb{S}'\setminus\{\bo\}$ admits a hyperk\"ahler metric which is positive definite when $R>0$ and has neutral signature when $R<0$.
\end{theorem}
In \cite{Sw}, a family of quoternionic K\"ahler metrics are constructed by solving an ODE on the spinor bundle with the zero section removed over a quoternionic K\"ahler manifold, and the hyperk\"ahler metric above can be obtained as a special (limiting ) case of the construction in 4 dimensions.

In this paper, we give an alternative proof to Theorem \ref{swann} by using spinor calculus. We introduce the second complex structure $\mathbb{J}$ on $\mathbb{S}'\setminus\{\bo\}$ in a similar way to that of the original complex structure $\mathbb{I}$ by Atiyah--Hitchin--Singer, and define the K\"ahler form 
$\wt\omega_{\mathbb{J}}$ of a hyperk\"ahler metric by using $\mathbb{J}$ and a holomorphic symplectic form $\wt\omega$. It is shown that the canonical hermitian norm $\|\pi\|^2=\sigma_{A'\overline{B'}}\pi^{A'}\overline{\pi^{B'}}$ on $\mathbb{S}'$ gives a K\"ahler potential (relative to the complex structure $\mathbb{I}$):
\[
\wt\omega_{\mathbb{J}}=i\partial\overline\partial \|\pi\|^2.
\]
The resulting hyperk\"ahler metric $\wt g$ can be expressed as
\[
\wt g=2\sigma_{A'\overline{B'}}\d\pi^{A'}\cdot\overline{\d\pi^{B'}}+\Lambda\|\pi\|^2 g_{ab}\theta^a\cdot\theta^b,
\]
where $\d\pi^{A'}:=d\pi^{A'}+\Gamma_{cB'}{}^{A'}\pi^{B'}\theta^c$ is the annihilator of the horizontal distribution on $\mathbb{S}'\setminus\{\bo\}$, and $\Lambda:=(1/24)R$ is a multiple of the scalar curvature of the base metric $g=g_{ab}\theta^a\cdot\theta^b$. 

Since $\wt g$ is Ricci-flat and hence $\|\pi\|^2$ solves the complex Monge--Amp\`ere equation, the  $(1, 1)$-form
\[
\omega_{\rm KE}:=i\partial\overline{\partial}\log\|\pi\|^2
\]
descends to the K\"ahler form of a K\"ahler-Einstein metric $g_{\rm KE}$ on $\mathbb{P}(\mathbb{S}')$; see \S\ref{Ricci-flat}. We prove in Theorem \ref{Kahler-Einstein} that in terms of the decomposition
\[
T\mathbb{P}(\mathbb{S}')=H\oplus T({\rm fiber}),
\]
where $H$ is the horizontal lift of $TX$, it can be written as
\[
g_{\rm KE}=(\Lambda g)\oplus g_{\rm FS}
\]
with $g_{\rm FS}$ being the Fubini--Study metric on the fibers.

To obtain a hyperk\"ahler ambient metric on $\wt{\mathcal{G}}:=\mathbb{S}'\setminus\{\bo\}$ for  twistor CR manifolds, we apply these general constructions to the anti self-dual Poincar\'e-Einstein manifold $(X\setminus\Sigma, g_+)$. 
However, since $g_+$ is singular along $\Sigma$, the obtained hyperk\"ahler metric $\wt g$ on $\wt{\mathcal{G}}|_{X\setminus\Sigma}$ also has a singularity along $\wt{\mathcal{G}}|_{\Sigma}$. To ``compactify'' this metric, we pull it back via the dilation $\d_{|r|}$ (and multiply it by the signature of $r$), where $r$ is an arbitrary defining function of $\Sigma\subset X$. 
It can be shown that the resulting hyperk\"ahler structure $(\wt g [r], \mathbb{I}_{r},  \mathbb{J}_{r},  \mathbb{K}_{r})$ extends to $\wt{\mathcal{G}}$, and that these complex structures agree with those constructed by using the metric $\overline{g}_+=r^2 g_+$ in place of $g_+$. Moreover, the K\"ahler potential of $\wt g[r]$ is given as a homogeneous defining function $\wt r$ of $\wt{\mathcal{G}}|_\Sigma\subset\wt{\mathcal{G}}$. Since a hyperk\"ahler metric is Ricci-flat, this implies that $\wt g[r]$ is an ambient metric for the twistor CR manifold $M\cong\mathbb{P}(\mathbb{S}')|_\Sigma$. 

The main result of this paper is summarized as follows:

\begin{theorem}\label{main-theorem}
Let $(\Sigma, [h])$ be an orientable real analytic conformal $3$-manifold which is the conformal infinity of an anti self-dual Poincar\'e-Einstein manifold $(X\setminus\Sigma, g_+)$, and $\wt{\mathcal{G}}=\mathbb{S}'\setminus\{\bo\}$ be the spinor bundle for a fixed spin structure of $(X, [\overline{g}_+])$ with the zero section removed. Then, for any defining function $r$ of $\Sigma\subset X$, there exist complex structures 
$\mathbb{I}_r, \mathbb{J}_r, \mathbb{K}_r=\mathbb{I}_r\mathbb{J}_r=-\mathbb{J}_r\mathbb{I}_r$ and 
a compatible hyperk\"ahler metric $\wt g[r]$ of neutral signature on $\wt{\mathcal{G}}$. Moreover, the K\"ahler form of $\wt g[r]$ with respect to $\mathbb{I}_r$ is written in the form $i\partial\overline{\partial}\wt r$ with a defining function $\wt r$ of $\wt{\mathcal{G}}|_\Sigma\subset\wt{\mathcal{G}}$ which is homogeneous of degree $(1, 1)$, where $\partial\overline{\partial}$ is with respect to $\mathbb{I}_r$. In particular, $\wt g[r]$ gives an ambient metric for the twistor CR manifold $M\cong\mathbb{P}(\mathbb{S}')|_\Sigma$.

For another defining function $\wh r=e^{\U}r$, the dilation
\[
\d_{e^\U}\colon (\wt{\mathcal{G}}, \wt g [\wh r], \mathbb{I}_{\wh r},  \mathbb{J}_{\wh r},  \mathbb{K}_{\wh r})\longrightarrow (\wt{\mathcal{G}}, \wt g[r], \mathbb{I}_{r},  \mathbb{J}_{r},  \mathbb{K}_{r})
\]
gives an isomorphism of hyperk\"ahler manifolds.
\end{theorem}

We remark that in fact the assumption of real analyticity of $(\Sigma, [h])$ in this theorem can be dropped since we can embed the twistor CR manifold to the twistor space $\mathbb{P}(\mathbb{S}')$ only by assuming that 
$(\Sigma, [h])$ is the conformal infinity of an anti self-dual Poincar\'e-Einstein manifold. By LeBrun's non-embeddablity theorem, this implies that the real analyticity follows automatically from the assumption (Theorem \ref{analyticity}). 

We can also adapt the construction of the K\"ahler-Einstein metric to this setting:
The complex structure on $\mathbb{P}(\mathbb{S}')$ induced by $\mathbb{I}_r$ is independent of the choice of $r$, and the $(1, 1)$-form
\[
\omega_{\rm CY}:=-i\partial\overline\partial\log|\wt r|
\]
gives the K\"ahler form of a Lorentzian K\"ahler-Einstein metric $g_{\rm CY}$ on $\mathbb{P}(\mathbb{S}')|_{X\setminus\Sigma}$. We call $g_{\rm CY}$ the smooth Cheng--Yau metric since this corresponds to the complete K\"ahler-Einstein metrics on bounded strictly pseudoconvex domains in $\mathbb{C}^{n+1}$ constructed by Cheng--Yau \cite{ChY}. By applying the theorem in the general case, we obtain the expression
\[
g_{\rm CY}=(-\Lambda g_+)\oplus(-g_{\rm FS})
\]
in terms the decomposition
\[
T\mathbb{P}(\mathbb{S}')|_{X\setminus\Sigma}=H\oplus T({\rm fiber});
\]
see Theorem \ref{cheng-yau}.

We conclude by briefly discussing higher dimensional analogues of our results. LeBrun's twistor construction of Poincar\'e-Einstein metrics was extended to higher dimensions by Biquard \cite{Bi}, where the conformal infinities are given by quaternionic contact manifolds. Moreover, associated to quaternionic contact structures, there exists a construction of twistor CR manifolds arising from the inclusion of symmetry groups  $\mathrm{Sp}(p+1, q+1)\subset \mathrm{SU}(2p+2, 2q+2)$; see \cite{CS}. These suggest that, for the Fefferman metrics of such twistor CR manifolds, one should be able to construct ambient metrics whose holonomy groups are contained in $\mathrm{Sp}(p+1, q+1)$.

\medskip

This paper is organized as follows: 
In \S\ref{spinor-calculus}, we summarize basic definitions and formulas in the spinor calculus needed in our constructions. In \S\ref{hyperkahler-spinor}, we define several complex structures on the spinor bundle $\mathbb{S}'\setminus\{\bo\}$ over an anti self-dual Einstein 4-manifold $(X, g)$ with non-zero scalar curvature, and construct a compatible hyperk\"ahler metric $\wt g$.  Then, we give a description of the associated K\"ahler-Einstein metric on the twistor space $\mathbb{P}(\mathbb{S}')$ in \S\ref{MA-Kahler-Einstein}. In \S\ref{ambient-conf-CR}, we review the notion of ambient metric and some related metrics in conformal and CR geometries. In \S\ref{hyperkahler-ambient}, we give the definition of twistor CR manifold and embed it into the twistor space of the anti self-dual Poincar\'e-Einstein manifold. Then, by adapting the general theory, we construct neutral hyperk\"ahler ambient metrics associated with twistor CR manifolds, and give a description of the smooth Cheng--Yau metric. Finally, in \S\ref{flat-case} we compute the hyperk\"ahler ambient metric for the flat model.

\bigskip

{\it Notation}:
We use Einstein's summation convention throughout the paper. The spinor indices are lowered and raised by the skew forms $\e_{AB}, \e_{A'B'}$ and their inverses $\e^{AB}, \e^{A'B'}$ as 
\[
\xi_A=\xi^B\e_{BA}, \quad \eta_{A'}=\eta^{B'}\e_{B'A'}, \quad \xi^A=\e^{AB}\xi_{B}, \quad 
\eta^{A'}=\e^{A'B'}\eta_{B'},
\]
except for \S\ref{flat-case} where we use the compactified skew form $\overline{\e}_{AB}$, etc. 

We define the (skew) symmetrizations of indices by
\[
\varphi_{(ab)}:=\frac{1}{2}(\varphi_{ab}+\varphi_{ba}), \quad \varphi_{[ab]}:=\frac{1}{2}(\varphi_{ab}-\varphi_{ba})
\]
and similarly for $k$-tensors or spinors, dividing by the factor $k!$.

For 1-forms $\varphi, \psi$, we set
\begin{align*}
(\varphi\cdot\psi)(V, W)&:=\frac{1}{2}\bigl(\varphi(V)\psi(W)+\varphi(W)\psi(V)\bigr), \\
(\varphi\wedge\psi)(V, W)&:=\varphi(V)\psi(W)-\varphi(W)\psi(V)
\end{align*}
so that the correspondence between K\"ahler metrics and K\"ahler forms is given by
\[
2g_{\a\overline{\b}}\theta^{\a}\cdot \overline{\theta^{\b}}\longleftrightarrow
ig_{\a\overline{\b}}\theta^{\a}\wedge\overline{\theta^{\b}}.
\]
\medskip

\noindent {\bf Acknowledgment} 
The author thanks the anonymous referees for their helpful comments.
This  work was partially supported by JSPS KAKENHI Grant Number 22K13922.

\section{The spinor calculus}\label{spinor-calculus}

We review basic definitions and formulas in the spinor calculus. We basically follow the notation in \cite{PR1}, although they deal with Lorentzian metrics and adopt different sign convention for curvature tensors from ours.  
\subsection{Riemannian spinor bundles}
Let $(X, g)$ be a 4-dimensional oriented Riemannian manifold, and let $\mathcal{F}\to X$  
be the oriented orthonormal frame bundle. A spin structure on $X$ is a principal $\mathrm{Spin}(4)$-bundle $\wt{\mathcal{F}}\to X$ which has a $\mathrm{Spin}(4)$-equivariant double covering $\wt{\mathcal{F}}\to\mathcal{F}$. Here, $\mathrm{Spin}(4)$ acts on $\mathcal{F}$ from the right through 
$\mathrm{Spin}(4)\to \mathrm{SO}(4)$. 

We recall that $(A, B)\in \mathrm{Spin}(4)\cong \mathrm{SU}(2)\times \mathrm{SU}(2)$ acts on 
\[
\mathbb{R}^4\cong\mathbb{M}:=\Bigl\{Y=\frac{1}{\sqrt{2}}
\begin{pmatrix}
x^0+ix^3 & x^1+ix^2 \\
-x^1+ix^2 & x^0-ix^3
\end{pmatrix}
\Big|\ 
x^a\in\mathbb{R}
\Bigr\}
\]
by
\[
(A, B)\cdot Y=AYB^{-1}.
\]
Since the standard metric on $\mathbb{M}$ is represented by $\langle Y, Y\rangle=2\det Y$, this action preserves the metric and gives the double covering map $\mathrm{Spin}(4)\to \mathrm{SO}(4)$. By using the representations $\rho_\pm$ of $\mathrm{Spin}(4)$ on $\mathbb{C}^2$ given by
\[
\rho_+(A, B)(\xi):=A\xi, \quad \rho_-(A, B)(\eta):={}^t\! B^{-1}\eta=\overline{B}\eta,
\]
we define the spinor bundles over $X$ by 
\[
\mathbb{S}:=\wt{\mathcal{F}}\times_{\rho_+}\mathbb{C}^2, \quad 
\mathbb{S}':=\wt{\mathcal{F}}\times_{\rho_-}\mathbb{C}^2.
\]
Then, since the representation $(\mathbb{C}^2\otimes\mathbb{C}^2, \rho_+\otimes\rho_-)$ is equivalent to $\mathbb{C}\otimes\mathbb{M}$, we have 
\begin{equation}\label{spinor-isom}
\mathbb{S}\otimes\mathbb{S}'\cong \mathbb{C}TX.
\end{equation}
If we use unprimed indices $A, B, C, \dots$ for $\mathbb{S}$ and primed indices $A', B', C', \dots$ for $\mathbb{S}'$, the isomorphism \eqref{spinor-isom} is represented as 
\[
v^a=\gamma^a_{AA'}v^{AA'}, \quad  v^{AA'}=\gamma_a^{AA'}v^a
\]
with a tensor $\gamma^a_{AA'}, \gamma_{a}^{AA'}$. We often abbreviate these expressions as $v^a=v^{AA'}$. The skew symmetric form on $\mathbb{C}^2$ expressed by the matrix
\begin{equation}\label{epsilon}
\e=\begin{pmatrix}
0 & 1 \\
-1 & 0
\end{pmatrix}
\end{equation}
is $\rho_{\pm}$-invariant, so it induces skew symmetric forms $\e=(\e_{AB}), \e'=(\e_{A'B'})$ on $\mathbb{S}, \mathbb{S}'$. Under the isomorphism \eqref{spinor-isom}, we have
\[
g_{ab}=\e_{AB}\e_{A'B'}
\]
on $\mathbb{C}TX$. The inverses of $\e, \e'$ are denoted by $\e^{AB}, \e^{A'B'}$: 
\[
\e_{AB}\e^{CB}=\d_A{}^C, \quad \e_{A'B'}\e^{C'B'}=\d_{A'}{}^{C'}.
\]
We lower and raise spinor indices with these skew forms:
\[
\xi_A=\xi^B\e_{BA}, \quad \eta_{A'}=\eta^{B'}\e_{B'A'}, \quad \xi^A=\e^{AB}\xi_{B}, \quad 
\eta^{A'}=\e^{A'B'}\eta_{B'}.
\]
We also have $\mathbb{C}$-anti-linear bundle maps
\[
\sigma\colon \mathbb{S}\longrightarrow\mathbb{S}, \quad 
\sigma\colon \mathbb{S}'\longrightarrow\mathbb{S}'
\]
induced by the $\rho_\pm$-equivariant map
\begin{equation}\label{sigma}
\sigma
\begin{pmatrix}
a \\
b
\end{pmatrix}=
\begin{pmatrix}
\overline{b} \\
-\overline{a}
\end{pmatrix}.
\end{equation}
We express these maps by 
\[
\xi^{A}\longmapsto \sigma^A{}_{\overline B}\overline{\xi^{B}}, \quad 
\eta^{A'}\longmapsto \sigma^{A'}{}_{\!\!\overline{B'}}\overline{\eta^{B'}}
\]
with $\sigma^A{}_{\overline B}\in \mathbb{S}\otimes\overline{\mathbb{S}}^*, 
\sigma^{A'}{}_{\!\!\overline{B'}}\in \mathbb{S}'\otimes\overline{\mathbb{S}'}^*$.
Since $\sigma^2=-{\rm id}$, we have
\begin{equation}\label{sigma-square}
\sigma^{A}{}_{\overline{C}}\,\sigma^{\overline C}{}_{B}=-\d_{B}{}^A, \quad 
\sigma^{A'}{}_{\!\!\overline{C'}}\,\sigma^{\overline{C'}}{}_{\!\!B'}=-\d_{B'}{}^{A'}, 
\end{equation}
where we set
\[
\sigma^{\overline A}{}_{B}:=\overline{\sigma^{A}{}_{\overline B}}, \quad \sigma^{\overline{A'}}{}_{\!\!B'}:=\overline{\sigma^{A'}{}_{\!\!\overline{B'}}}.
\]
We also have 
\[
\e(\sigma(\xi_1), \sigma(\xi_2))=\overline{\e(\xi_1, \xi_2)}, \quad 
\e'(\sigma(\eta_1), \sigma(\eta_2))=\overline{\e'(\eta_1, \eta_2)},
\]
which implies
\begin{equation}\label{invariance-e-tensor}
\sigma_{B\overline{C}}\,\sigma^{B}{}_{\overline D}=\overline{\e_{CD}}, \quad
\sigma_{B'\overline{C'}}\sigma^{B'}{}_{\!\!\overline{D'}}=\overline{\e_{C'D'}}.
\end{equation}
Contracting these equations with $\sigma^{\overline C}{}_{A}$, $\sigma^{\overline{C'}}{}_{\!\!A'}$, we have
\begin{equation}\label{sigma-sym}
\sigma_{A\overline{D}}=\sigma_{\overline{D}A}, \quad \sigma_{A'\overline{D'}}
=\sigma_{\overline{D'}A'}.
\end{equation}
The hermitian forms
\[
\e(\sigma(\xi_2), \xi_1)=\sigma_{A\overline{B}}\xi_1^A\overline{\xi_2^B}, \quad 
\e'(\sigma(\eta_2), \eta_1)=\sigma_{A'\overline{B'}}\eta_1^{A'}\overline{\eta_2^{B'}}
\]
are positive definite.

Under the isomorphism \eqref{spinor-isom}, the complex conjugation $\mathbb{C}TX\to \mathbb{C}TX$ corresponds to the map $\sigma\otimes\sigma\colon \mathbb{S}\otimes\mathbb{S}'\to \mathbb{S}\otimes\mathbb{S}'$,
\[
v^a=v^{AA'}\longmapsto \overline{v}^a=\sigma^{A}{}_{\overline B}\sigma^{A'}{}_{\!\!\overline{B'}}\overline{v^{BB'}}.
\]
This should not  be confused with the complex conjugate $\overline{v^{AA'}}\in \overline{\mathbb{S}}\otimes\overline{\mathbb{S}'}$ as a spinor. For a covector 
$\varphi_a=\varphi_{AA'}$, we have
\[
\overline{\varphi}_a v^a=\overline{\varphi_a\overline{v}^a}=
\overline{\varphi_{AA'}\sigma^{A}{}_{\overline B}\sigma^{A'}{}_{\!\!\overline{B'}}\overline{v^{BB'}}}=\sigma^{\overline A}{}_{B}\sigma^{\overline{A'}}{}_{\!\!B'}\overline{\varphi_{AA'}}v^{BB'}.
\]
Thus, the complex conjugate is given by
\[
\overline{\varphi}_a=\sigma^{\overline B}{}_{A}\sigma^{\overline{B'}}{}_{\!\!A'}\overline{\varphi_{BB'}}=\sigma_{A}{}^{\overline B}\sigma_{A'}{}^{\overline{B'}}\overline{\varphi_{BB'}}.
\]

\subsection{The curvature spinors}
The spinor bundles admit a linear connection $\nabla$ induced by the principal connection on $\wt{\mathcal{F}}$, which is the lift of the Levi-Civita connection on $\mathcal{F}$. By this definition, the isomorphism \eqref{spinor-isom} preserves the connection:
\[
\nabla_c\gamma^b_{AA'}=0, \quad \nabla_c\gamma_b^{AA'}=0.
\]
Moreover, the skew forms $\e, \e'$ and the bundle map $\sigma$ are parallel with respect to this connection:
\[
\nabla_c\e_{AB}=0, \quad \nabla_c\e_{A'B'}=0, \quad \nabla_c\sigma^A{}_{\overline B}=0, \quad 
\nabla_c\sigma^{A'}{}_{\overline{B'}}=0.
\]
In fact, the connection $\nabla$ is characterized by the condition that it preserve the skew forms $\e, \e'$ and induce the Levi-Civita connection on $\mathbb{S}\otimes\mathbb{S}'\cong \mathbb{C}TX$.

Let $R_{abcd}, R_{ab}, R$ be the curvature tensor, the Ricci tensor, and the scalar curvature of $g$: 
$(\nabla_a\nabla_b-\nabla_b\nabla_a)v^c=R_{ab}{}^c{}_d v^d, \ R_{ab}=R_{ca}{}^c{}_b, \  R=R_a{}^a$. (Note that our sign convention is opposite to that in \cite{PR1}.) 
Then we have a decomposition
\begin{align*}
R_{abcd}&=\Psi_{ABCD}\e_{A'B'}\e_{C'D'}+\wt\Psi_{A'B'C'D'}\e_{AB}\e_{CD} \\
&\quad +\Phi_{ABC'D'}\e_{A'B'}\e_{CD}+\Phi_{CDA'B'}\e_{AB}\e_{C'D'} \\
&\quad +\Lambda(\e_{AC}\e_{BD}\e_{A'B'}\e_{C'D'}+\e_{AD}\e_{BC}\e_{A'B'}\e_{C'D'} \\
&\quad\quad\quad +\e_{A'C'}\e_{B'D'}\e_{AB}\e_{CD}+\e_{A'D'}\e_{B'C'}\e_{AB}\e_{CD})
\end{align*}
with spinors 
\begin{align*}
\Psi_{ABCD}=\Psi_{(ABCD)},\quad  \wt\Psi_{A'B'C'D'}=\wt\Psi_{(A'B'C'D')}, \quad  \Phi_{ABC'D'}=\Phi_{(AB)(C'D')}
\end{align*}
and a scalar $\Lambda$. These spinor curvatures satisfy
\begin{align*}
W_{abcd}&=\Psi_{ABCD}\e_{A'B'}\e_{C'D'}+\wt\Psi_{A'B'C'D'}\e_{AB}\e_{CD}, \\
 R_{(ab)_0}&=-2\Phi_{ABA'B'}, \\
 R&=24\Lambda,
\end{align*}
where 
\[
W_{abcd}=R_{abcd}-2P_{c[a}g_{b]d}+2P_{d[a}g_{b]c}\quad 
\Bigl(P_{ab}=\frac{1}{2}R_{ab}-\frac{1}{12}R g_{ab}\Bigr)
\]
is the Weyl curvature and $R_{(ab)_0}$ is the trace-free part of $R_{ab}$.

The spinor expression of the volume form
\[
\frac{1}{4!}e_{abcd}\th^a\wedge\th^b\wedge\th^c\wedge\th^d, \quad e_{abcd}=e_{[abcd]}
\]
is given by
\[
e_{abcd}=\e_{AC}\e_{BD}\e_{A'D'}\e_{B'C'}-\e_{AD}\e_{BC}\e_{A'C'}\e_{B'D'}.
\]
From this, we have
\[
W^-_{abcd}=\Psi_{ABCD}\e_{A'B'}\e_{C'D'}, \quad W^+_{abcd}=\wt\Psi_{A'B'C'D'}\e_{AB}\e_{CD},
\]
where $W^-$ is the anti self-dual part and $W^+$ is the self-dual part of the Weyl curvature: $\frac{1}{2}e_{cd}{}^{pq}W^{\pm}_{abpq}=\pm W^{\pm}_{abcd}$.

The curvature $\Omega_{ab}{}^{C'}{}_{D'}$ of $\nabla$ on $\mathbb{S}'$ is given by
\begin{align*}
\Omega_{ab}{}^{C'}{}_{D'}\pi^{D'}
&=(\nabla_a\nabla_b-\nabla_b\nabla_a)\pi^{C'} \\
&=-\Phi_{AB}{}^{C'}{}_{D'}\pi^{D'}\e_{A'B'}-\wt\Psi_{A'B'}{}^{C'}{}_{D'}\pi^{D'}\e_{AB}+2\Lambda\d_{(A'}{}^{C'}\pi_{B')}\e_{AB}.
\end{align*}
In particular, if $g$ is anti self-dual ($\wt\Psi_{A'B'C'D'}=0$) and Einstein ($\Phi_{ABC'D'}=0$), we have
\begin{equation}\label{S-prime-curvature}
\Omega_{ab}{}^{C'}{}_{D'}\pi^{D'}=2\Lambda\e_{AB}\d_{(A'}{}^{C'}\pi_{B')}.
\end{equation}

Let $\frac{1}{2}\Omega^*_{ab}{}^{\overline{C'}}{}_{\overline{D'}}\theta^a\wedge\theta^b$ be the curvature form of $\nabla$ on $\overline{\mathbb{S}'}$. Then, we have
\[
\Omega^*_{ab}{}^{\overline{C'}}{}_{\overline{D'}}\theta^a\wedge\theta^b
=\overline{\Omega_{ef}{}^{C'}{}_{D'}\theta^e\wedge\theta^f},
\]
and hence
\[
\Omega^*_{ab}{}^{\overline{C'}}{}_{\overline{D'}}=\sigma_A{}^{\overline E}\sigma_{A'}{}^{\overline{E'}}\sigma_{B}{}^{\overline F}\sigma_{B'}{}^{\overline{F'}}\,\overline{\Omega_{EE'FF'}{}^{C'}{}_{D'}}.
\]
It follows from this identity together with \eqref{S-prime-curvature}, \eqref{invariance-e-tensor} that
\begin{equation}\label{Omega-star-pi}
\begin{aligned}
\Omega^*_{ab}{}^{\overline{C'}}{}_{\overline{D'}}\overline{\pi^{D'}}&=
2\Lambda\sigma_A{}^{\overline E}\sigma_{A'}{}^{\overline{E'}}\sigma_B{}^{\overline F}\sigma_{B'}{}^{\overline{F'}}\,\overline{\d_{(E'}{}^{C'}\pi_{F')}\e_{EF}} \\
&=\Lambda\sigma_{A\overline{F}}\sigma_{A'}{}^{\overline{C'}}\sigma_B{}^{\overline F}\sigma_{B'}{}^{\overline{F'}}\overline{\pi_{F'}}+\Lambda\sigma_{A\overline{F}}
\sigma_{A'}{}^{\overline{E'}}\sigma_{B}{}^{\overline F}\sigma_{B'}{}^{\overline{C'}}\overline{\pi_{E'}} \\
&=2\Lambda\e_{AB}\sigma_{(A'}{}^{\overline{C'}}\sigma_{B')}{}^{\overline{E'}}\overline{\pi_{E'}}
\ \Bigl(=-2\Lambda\e_{AB}\sigma_{(A'}{}^{\overline{C'}}(\sigma\pi)_{B')}\Bigr),
\end{aligned}
\end{equation}
where we set
\[
(\sigma\pi)^{A'}:=\sigma^{A'}{}_{\overline{B'}}\overline{\pi^{B'}}.
\]


\subsection{Conformal spinor bundles}
For the spinor calculus on conformal manifolds, we consider the groups 
\[
G:=\mathrm{SO}(4)\times\mathbb{R}_+,\quad \wt G:=\mathrm{Spin}(4)\times\mathbb{R}_+
=\mathrm{SU}(2)\times \mathrm{SU}(2)\times\mathbb{R}_+
\]
and the map $\wt G\to G$ defined by 
\begin{equation*}
(u, a)\longmapsto (\pi_0(u), a^2),
\end{equation*}
where $\pi_0\colon \mathrm{Spin}(4)\to \mathrm{SO}(4)$ is the double covering map. Then, $\wt G$ acts on $\mathbb{M}$ by 
\[
(A, B, a)\cdot Y=a^2 AYB^{-1}.
\]

Let $(X, [g])$ be a 4-dimensional oriented conformal manifold. Then the frame bundle over $X$ has a reduction to a $G$-bundle $\mathcal{F}\to X$. If a representative metric $g\in[g]$ admits a spin structure, then we have a $\wt G$-bundle 
$\wt{\mathcal{F}}\to X$ and a $G$-equivariant double covering map $\pi\colon \wt{\mathcal{F}}\to\mathcal{F}$, where $\wt G$ acts on $\mathcal{F}$ from the right through the map $\wt G\to G$. By using the representations $\rho^c_\pm$ of $\wt G$ on $\mathbb{C}^2$ given by
\[
\rho^c_+(A, B, a)(\xi):=a A\xi, \quad \rho^c_-(A, B, a)(\eta):=a{}^t\! B^{-1}\eta=a\overline{B}\eta,
\]
we define the spinor bundles over $X$ by 
\[
\mathbb{S}:=\wt{\mathcal{F}}\times_{\rho^c_+}\mathbb{C}^2, \quad 
\mathbb{S}':=\wt{\mathcal{F}}\times_{\rho^c_-}\mathbb{C}^2.
\]
Then, since the representation $(\mathbb{C}^2\otimes\mathbb{C}^2, \rho^c_+\otimes\rho^c_-)$ is equivalent to $\mathbb{C}\otimes\mathbb{M}$, we have $\mathbb{S}\otimes\mathbb{S}'\cong \mathbb{C}TX$. As before, we use unprimed indices $A, B, C, \dots$ for $\mathbb{S}$ and primed indices $A', B', C', \dots$ for $\mathbb{S}'$. The conformal classes of the skew forms $[\e_{AB}], [\e_{A'B'}]$ determined by \eqref{epsilon} are well-defined and we have 
$[g_{ab}]=[\e_{AB}\e_{A'B'}]$. We also have the $\mathbb{C}$-anti-linear bundle maps
\[
\sigma=(\sigma^{A}{}_{\overline{B}})\colon \mathbb{S}\longrightarrow\mathbb{S}, \quad 
\sigma=(\sigma^{A'}{}_{\!\!\overline{B'}})\colon \mathbb{S}'\longrightarrow\mathbb{S}'
\]
induced by the $\rho^c_{\pm}$-equivariant map \eqref{sigma}.

Each choice of a representative metric $g\in[g]$ determines a reduction of $\mathcal{F}$ to an $\mathrm{SO}(4)$-bundle $\mathcal{F}_g\subset \mathcal{F}$ and this defines a $\mathrm{Spin}(4)$-bundle $\wt{\mathcal{F}}_g:=\pi^{-1}(\mathcal{F}_g)\subset \wt{\mathcal{F}}$. Then, the spinor bundles can be represented as 
\[
\mathbb{S}=\wt{\mathcal{F}}_g\times_{\rho_+}\mathbb{C}^2, \quad 
\mathbb{S}'=\wt{\mathcal{F}}_g\times_{\rho_-}\mathbb{C}^2
\]
and we obtain skew forms $\e_{AB}\in[\e_{AB}], \e_{A'B'}\in[\e_{A'B'}]$ such that 
$g_{ab}=\e_{AB}\e_{A'B'}$. Let $\wh g=e^{2\U}g\ (\U\in C^\infty(X))$ be another representative metric.
Then we have 
\[
\mathcal{F}_{\wh g}=\mathcal{F}_g\cdot (I_4, e^{-\U})
\]
in $\mathcal{F}$, and hence
\[
\wt{\mathcal{F}}_{\wh g}=\wt{\mathcal{F}}_g\cdot (I_2, I_2, e^{-\frac{\U}{2}})
\]
in $\wt{\mathcal{F}}$. This implies that if a spinor $\xi\in\mathbb{S}$ has fiber coordinates ${}^t(a, b)$ in a frame in $\wt{\mathcal{F}}_g$, it has coordinates $e^{\U/2}\cdot{}^t(a, b)$ in a frame in $\wt{\mathcal{F}}_{\wh g}$, and similarly for $\mathbb{S}'$. Hence we have
\begin{equation}\label{e-rescale}
\wh\e_{AB}=e^{\U}\e_{AB}, \quad \wh\e_{A'B'}=e^{\U}\e_{A'B'}.
\end{equation}

The conformal transformation formulas for the connections on spinor bundles are given by
\begin{equation}\label{spinor-conf}
\begin{aligned}
\wh\nabla_a \xi^B&=\nabla_a \xi^B+\d_A{}^B \U_{CA'}\xi^C, \\
\wh\nabla_a \eta^{B'}&=\nabla_a \eta^{B'}+\d_{A'}{}^{B'} \U_{AC'}\eta^{C'},
\end{aligned}
\end{equation}
where $\U_a=\U_{AA'}$ is defined by $d\U=\U_a\theta^a$. Indeed, one can show that  
these connections preserve $\wh\e_{AB}, \wh\e_{A'B'}$ and induce the Levi-Civita connection of $\wh g$ on $\mathbb{S}\otimes\mathbb{S}'\cong \mathbb{C}TX$.
We note that the bundle map $\sigma$ is conformally invariant, and parallel with respect to the connection determined by any representative metric.


\section{The hyperk\"ahler metric on $\mathbb{S}'\setminus\{\bo\}$}\label{hyperkahler-spinor}

\subsection{The complex structure $\mathbb{I}$ on $\mathbb{S}'\setminus\{\bo\}$}\label{def-I}
Let $g$ be an anti self-dual Einstein metric with $\Lambda\neq0$ on a 4-dimensional manifold $X$.
We assume that $(X, g)$ admits a spin structure and let $\mathbb{S}, \mathbb{S}'$ be the spinor bundles over $X$ for a fixed spin structure. For each $[\pi]=[\pi^{A'}]\in \mathbb{P}(\mathbb{S}'_x)$, the 2-dimensional totally null subspace
\[
\mathcal{A}_{[\pi]}:=\Bigl\{\xi^A\pi^{A'}\frac{\partial}{\partial x^a}\in \mathbb{C}T_x X\ \Big|\ \xi^A\in \mathbb{S}_x\Bigr\}
\]
is called an {\it $\a$-plane}. Since the complex conjugation $\mathbb{C}TX\cong \mathbb{S}\otimes\mathbb{S}'\to \mathbb{S}\otimes\mathbb{S}'\cong\mathbb{C}TX$ is given by $\sigma\otimes\sigma$, we have
\[
\overline{\mathcal{A}_{[\pi]}}=\mathcal{A}_{[\sigma(\pi)]}.
\]
Moreover, since $\e'(\sigma(\pi), \pi)>0$, the spinors $\pi, \sigma(\pi)$ form a basis of $\mathbb{S}'_x$ and hence $\mathcal{A}_{[\pi]}\cap\mathcal{A}_{[\sigma(\pi)]}=\{0\}$. Thus we have
\[
\mathbb{C}T_x X=\mathcal{A}_{[\pi]}\oplus\mathcal{A}_{[\sigma(\pi)]}=\mathcal{A}_{[\pi]}\oplus\overline{\mathcal{A}_{[\pi]}}.
\]
Then we can define an almost complex structure $\mathbb{I}$ on $\mathbb{S}'\setminus\{\bo\}$ by 
setting
\begin{align*}
T^{0, 1}_\pi (\mathbb{S}'\setminus\{\bo\})&:=\wt{\mathcal{A}}_{[\pi]}\oplus T^{0,1}_\pi({\rm fiber}), \\ 
T^{1, 0}_\pi(\mathbb{S}'\setminus\{\bo\})&:=\wt{\mathcal{A}}_{[\sigma(\pi)]}\oplus T^{1, 0}_\pi({\rm fiber}),
\end{align*}
where $\wt{\mathcal{A}}_{[\pi]}, \wt{\mathcal{A}}_{[\sigma(\pi)]}$ denote the horizontal lifts of $\mathcal{A}_{[\pi]}, \mathcal{A}_{[\sigma(\pi)]}$ with respect to the connection $\nabla$ on $\mathbb{S}'$ induced by the Levi-Civita connection of $g$.

Let $(\Gamma_{cA'}{}^{B'}dx^c)$ be the connection form of $\nabla$. If we set
\[
\Gamma_{c\overline{A'}}{}^{\overline{B'}}:=\overline{\Gamma_{cA'}{}^{B'}},
\]
then the horizontal lift of a vector $v=v^a(\partial/\partial x^a)\in\mathbb{C}TX$ to $\mathbb{C}T_\pi(\mathbb{S}'\setminus\{\bo\})$ is written as
\[
\wt v=v^a\frac{\partial}{\partial x^a}-v^c\Bigl(\Gamma_{cA'}{}^{B'}\pi^{A'}\frac{\partial}{\partial \pi^{B'}}+\Gamma_{c\overline{A'}}{}^{\overline{B'}}\overline{\pi^{A'}}\frac{\partial}{\partial \overline{\pi^{B'}}}\Bigr).
\]
Hence, for any local frame $(\theta^a)$ of $\mathbb{C}T^*X$, the set of 1-forms
\[
\d\pi^{A'}:=d\pi^{A'}+\Gamma_{cB'}{}^{A'}\pi^{B'}\theta^c, \quad \pi_{A'}\theta^{AA'}\bigl(=\pi_{A'}\gamma^{AA'}_b\theta^b\bigr)
\] 
provides a local frame for $T^*_{1, 0}(\mathbb{S}'\setminus\{\bo\})$.

The following is a fundamental theorem in twistor theory:

\begin{theorem}[\cite{P, AHS}]\label{I-integrable}
The almost complex structure $\mathbb{I}$ is integrable.
\end{theorem}
\begin{proof}
We will show that $d(\d\pi^{A'})$ and $d(\pi_{A'}\theta^{AA'})$ do not contain $(0, 2)$-parts.
We fix a point $x\in X$ and take local frames of $\mathbb{S}, \mathbb{S}'$ such that $\Gamma_{cB}{}^{A}(x)=\Gamma_{cB'}{}^{A'}(x)=0$. These frames determine a local coframe $(\theta^a)$ on $X$ via the isomorphism $\mathbb{C}TX\cong\mathbb{S}\otimes\mathbb{S}'$. Then we have
\begin{equation}\label{at-x}
d\e_{A'B'}=0, \quad d\theta^a=0, \quad d\gamma^{AA'}_b=0, \quad 
d(\Gamma_{eB'}{}^{A'}\theta^e)=\frac{1}{2}\Omega_{ef}{}^{A'}{}_{B'}\theta^e\wedge\theta^f
\end{equation}
at $x$. Hence, by \eqref{S-prime-curvature} we have
\begin{equation}\label{d-delta}
\begin{aligned}
d(\d\pi^{A'})&=\frac{1}{2}\Omega_{ef}{}^{A'}{}_{B'}\pi^{B'}\theta^e\wedge\theta^f 
=\Lambda\e_{EF}\d_{(E'}{}^{A'}\pi_{F')}\theta^e\wedge\theta^f \\
&=\Lambda\e_{EF}\d_{E'}{}^{A'}\pi_{F'}\theta^e\wedge\theta^f \equiv 0\quad {\rm mod}\ \pi_{F'}\theta^{FF'}, \\
d(\pi_{A'}\theta^{AA'})&=\d\pi_{A'}\wedge \theta^{AA'}\equiv 0\quad {\rm mod}\ \d\pi^{B'}
\end{aligned}
\end{equation}
at $x$. Thus, $\mathbb{I}$ is integrable.
\end{proof}
\begin{rem}
As is proved in \cite{P, AHS}, the integrability of $\mathbb{I}$ follows solely from the anti self-duality of $g$. In fact, if we only have $\Psi_{A'B'C'D'}=0$, the derivative $d(\d\pi^{A'})$ above has the extra term  $-(1/2)\Phi_{EF}{}^{A'}{}_{B'}\pi^{B'}\e_{E'F'}\theta^e\wedge\theta^f$, but the contraction of this form with a vector in $T^{0, 1}$ becomes a $(1, 0)$-form and hence $d(\d\pi^{A'})$ does not have $(0, 2)$-part.
\end{rem}

For each $\pi\in \mathbb{S}'_x\setminus\{0\}$, we set
\[
\|\pi\|^2:= (\sigma\pi)_{A'}\pi^{A'}=\sigma_{A'\overline{B'}}\pi^{A'}\overline{\pi^{B'}}\,(>0)
\]
and define 
\[
I_{B'}{}^{A'}:=\frac{-i}{\|\pi\|^2}\bigl(\pi_{B'}(\sigma\pi)^{A'}+(\sigma\pi)_{B'}\pi^{A'}\bigr),
\]
which satisfies
\begin{equation}\label{I-map}
I_{B'}{}^{A'}\pi^{B'}=-i\pi^{A'}, \quad I_{B'}{}^{A'}(\sigma\pi)^{B'}=i(\sigma\pi)^{A'}.
\end{equation}
Since $\pi^{A'}$ and $(\sigma\pi)^{A'}$ form a basis of $\mathbb{S}'_x$, these equations characterize $I_{B'}{}^{A'}$. We define an endomorphism
\[
I_\pi\colon \mathbb{C}T_x X\longrightarrow \mathbb{C}T_x X
\]
by
\[
v^{AA'}\frac{\partial}{\partial x^a}\longmapsto v^{AB'}I_{B'}{}^{A'}\frac{\partial}{\partial x^a}.
\]
Then, \eqref{I-map} implies that $\mathcal{A}_{[\pi]}, \mathcal{A}_{[\sigma(\pi)]}$ are eigenspaces of $I_\pi$ with eigenvalues $-i, i$ respectively. Thus, $I_\pi$ gives an almost complex structure on $T_x X$, and in terms of the decomposition
\begin{equation}\label{T-decom}
T_\pi(\mathbb{S}'\setminus\{\bo\})=\wt{T_x X}\oplus T_\pi({\rm fiber})\cong T_x X\oplus \mathbb{S}'_x,
\end{equation}
the complex structure $\mathbb{I}$ can be expressed as
\[
\mathbb{I}=I_\pi \oplus I_{\mathbb{S}'_x},
\]
where $I_{\mathbb{S}'_x}$ is the standard complex structure on the fiber $\mathbb{S}'_x$.


\subsection{Another complex structure $\mathbb{J}$ on $\mathbb{S}'\setminus\{\bo\}$}\label{def-J}
To construct a hyperk\"ahler structure on $\mathbb{S}'\setminus\{\bo\}$, we define another complex structure $\mathbb{J}$ which anti-commutes with $\mathbb{I}$.

First, for each $\pi\in\mathbb{S}'_x\setminus\{0\}$, we define an endomorphism $J_\pi\colon \mathbb{C}T_x X\to \mathbb{C}T_x X$
by
\[
v^{AA'}\frac{\partial}{\partial x^a}\longmapsto v^{AB'}J_{B'}{}^{A'}\frac{\partial}{\partial x^a},
\]
where 
\[
J_{B'}{}^{A'}:=\frac{-1}{\|\pi\|^2}
\bigl(\pi_{B'}\pi^{A'}+(\sigma\pi)_{B'}(\sigma\pi)^{A'}\bigr).
\]
\begin{prop}
The endomorphism $J_\pi$ is a real operator and satisfies 
\begin{equation}\label{I-J}
J_\pi^2=-{\rm id}, \quad I_\pi J_\pi=-J_\pi I_\pi.
\end{equation}
\end{prop}
\begin{proof}
For the basis $(\pi, \sigma(\pi))$ of $\mathbb{S}'_x$, it holds that
\begin{equation}\label{J-map}
J_{B'}{}^{A'}\pi^{B'}=-(\sigma \pi)^{A'}, \quad J_{B'}{}^{A'}(\sigma\pi)^{B'}=\pi^{A'}.
\end{equation}
By \eqref{I-map} and \eqref{J-map}, we have 
\[
J_{B'}{}^{C'}J_{C'}{}^{A'}=-\d_{B'}{}^{A'},\quad  I_{B'}{}^{C'}J_{C'}{}^{A'}=-J_{B'}{}^{C'}I_{C'}{}^{A'}.
\]
Hence we obtain \eqref{I-J}. It also follows from \eqref{J-map} that $J_{B'}{}^{A'}$ commutes with $\sigma\colon \mathbb{S}'_x\to \mathbb{S}'_x$. Thus, $J_\pi$ commutes with the complex conjugation $\sigma\otimes\sigma\colon \mathbb{C}T_xX\to \mathbb{C}T_xX$ and becomes a real operator. 
\end{proof}
By using the decomposition \eqref{T-decom}, we define an almost complex structure $\mathbb{J}$ on $\mathbb{S}'\setminus\{\bo\}$ by
\[
\mathbb{J}:=J_\pi\oplus \sigma.
\]
Since $J_\pi$ anti-commutes with $I_\pi$ by \eqref{I-J} and $\sigma$ is $\mathbb{C}$-anti-linear, we have
\[
\mathbb{I}\mathbb{J}=-\mathbb{J}\mathbb{I}.
\]
The $(0, 1)$-subspace of $\mathbb{C}T_\pi(\mathbb{S}'\setminus\{\bo\})$ with respect to $\mathbb{J}$ is given by
\begin{align*}
{}^{\mathbb{J}}T^{0, 1} _\pi(\mathbb{S}'\setminus\{\bo\})
&=\wt{\mathcal{A}}_{[\pi-i\sigma(\pi)]}\oplus\{1\otimes\eta+i\otimes\sigma(\eta)\ |\ \eta\in{\mathbb{S}'}_x\} \\
&\subset\mathbb{C}\wt{T_x X}\oplus (\mathbb{C}\otimes_{\mathbb{R}}\mathbb{S}'_x)\cong
\mathbb{C}T_\pi(\mathbb{S}'\setminus\{\bo\}).
\end{align*}
Note that the isomprphism $\mathbb{S}'_x\cong T_\pi(\rm fiber)$ is given by 
\[
\eta^{A'}\longleftrightarrow \eta^{A'}\frac{\partial}{\partial \pi^{A'}}+\overline{\eta^{A'}}\frac{\partial}{\partial \overline{\pi^{A'}}}
\]
and hence $1\otimes\eta+i\otimes\sigma(\eta)\in \mathbb{C}\otimes_{\mathbb{R}}\mathbb{S}'_x$ corresponds to the complex vector
\[
\eta^{A'}\frac{\partial}{\partial \pi^{A'}}+\overline{\eta^{A'}}\frac{\partial}{\partial \overline{\pi^{A'}}}
+i\Bigl((\sigma\eta)^{A'}\frac{\partial}{\partial \pi^{A'}}+\overline{(\sigma\eta)^{A'}}\frac{\partial}{\partial \overline{\pi^{A'}}}\Bigr).
\]
Thus, the set of 1-forms
\[
\d\pi^{A'}\!\!-i\sigma^{A'}{}_{\!\!\overline{B'}}\overline{\d\pi^{B'}}, \quad 
\bigl(\pi_{A'}-i(\sigma\pi)_{A'}\bigr)\theta^{AA'}
\]
provides a local frame for ${}^\mathbb{J}T^*_{1, 0}(\mathbb{S}'\setminus\{\bo\})$ as they annihilate 
${}^{\mathbb{J}}T^{0, 1}(\mathbb{S}'\setminus\{\bo\})$.

The integrability of $\mathbb{J}$ can be shown in a similar way to the case of $\mathbb{I}$ (see Theorem \ref{J-parallel} for an alternative proof):
\begin{theorem}\label{J-integrable}
The almost complex structure $\mathbb{J}$ is integrable.
\end{theorem}
\begin{proof}
We fix a point $x\in X$ and take local frames for $\mathbb{S}, \mathbb{S}'$ as in the proof of Theorem \ref{I-integrable}. Note that in addition to \eqref{at-x} we also have $d\sigma^{A'}{}_{\!\!\overline{B'}}(x)=0$ since $\sigma$ is parallel. Then, by \eqref{Omega-star-pi} we have
\begin{equation}\label{d-sigma-delta}
\begin{aligned}
d\bigl(\sigma^{A'}{}_{\!\!\overline{B'}}\overline{\d\pi^{B'}}\bigr)
&=\sigma^{A'}{}_{\!\!\overline{B'}}
\frac{1}{2}\Omega^*_{ef}{}^{\overline{B'}}{}_{\overline{C'}}\overline{\pi^{C'}}\theta^e\wedge\theta^f \\
&=\Lambda\e_{EF}\sigma^{A'}{}_{\!\!\overline{B'}}\sigma_{(E'}{}^{\overline{B'}}(\sigma\pi)_{F')}\theta^e\wedge\theta^f \\
&=\Lambda\e_{EF}\d_{(E'}{}^{A'}(\sigma\pi)_{F')}\theta^e\wedge\theta^f \\
&=\Lambda\e_{EF}\d_{E'}{}^{A'}(\sigma\pi)_{F'}\theta^e\wedge\theta^f.
\end{aligned}
\end{equation}
By \eqref{d-delta} and \eqref{d-sigma-delta}, we have
\begin{align*}
d(\d\pi^{A'}\!\!-i\sigma^{A'}{}_{\!\!\overline{B'}}\overline{\d\pi^{B'}})
&=\Lambda\e_{EF}\d_{E'}{}^{A'}\bigl(\pi_{F'}
-i(\sigma\pi)_{F'}\bigr)\theta^e\wedge\theta^f \\
&\equiv 0\quad {\rm mod}\ \bigl(\pi_{F'}
-i(\sigma\pi)_{F'}\bigr)\theta^{FF'}.
\end{align*}
We also have
\begin{align*}
d\bigl(\bigl(\pi_{A'}-i(\sigma\pi)_{A'}\bigr)\theta^{AA'}\bigr)&=
\bigl(\d\pi_{A'}-i\sigma_{A'\overline{B'}}\overline{\d\pi^{B'}}\bigr)\wedge\theta^{AA'} \\
&\equiv 0 \quad {\rm mod}\ \d\pi^{C'}\!\!-i\sigma^{C'}{}_{\!\!\overline{B'}}\overline{\d\pi^{B'}}.
\end{align*}
Thus, $\mathbb{J}$ is integrable.
\end{proof}
\begin{rem}
As in the case of $\mathbb{I}$, one can prove the integrability of $\mathbb{J}$ only by the anti self-duality of $g$.
\end{rem}


\subsection{The hyperk\"ahler metric $\wt g$ on $\mathbb{S}'\setminus\{\bo\}$}\label{hyper}
We define a 1-form $\wt\tau$ on $\mathbb{S}'\setminus\{\bo\}$ by
\[
\wt\tau:=\pi_{B'}\d\pi^{B'}=\e_{A'B'}\pi^{A'}\d\pi^{B'}.
\]
Then, $\tau$ is a $(1, 0)$-form with respect to $\mathbb{I}$, and by a similar computation to the proof of 
Theorem \ref{I-integrable}, the exterior derivative 
\[
\wt\omega:=d\wt\tau
\]
is given by
\begin{equation}\label{omega}
\wt\omega=\e_{A'B'}\d\pi^{A'}\wedge\d\pi^{B'}+\Lambda\e_{AB}\pi_{A'}\pi_{B'}\theta^a\wedge\theta^b.
\end{equation}
Since this is a $(2, 0)$-form with respect to $\mathbb{I}$, both $\wt\tau$ and $\wt\omega$ are holomorphic with respect to $\mathbb{I}$. 

Next we define a 2-form $\wt\omega_{\mathbb{J}}$ on $\mathbb{S}'\setminus\{\bo\}$ by 
\begin{equation}
\wt\omega_{\mathbb{J}}(V, W):=\frac{-i}{2}\bigl(\wt\omega(V, \mathbb{J} W)-\wt\omega(W, \mathbb{J} V)\bigr).
\end{equation}
Note that for a 2-form $\varphi\wedge\psi$, we have
\begin{equation}\label{2-form-J}
(\varphi\wedge\psi)(V, \mathbb{J} W)-(\varphi\wedge\psi)(W, \mathbb{J} V)
=\bigl((\varphi\circ\mathbb{J})\wedge\psi+\varphi\wedge(\psi\circ\mathbb{J})\bigr)(V, W).
\end{equation}

\begin{prop}
We have
\begin{equation}\label{omega-J}
\wt\omega_{\mathbb{J}}=i\sigma_{A'\overline{B'}}\d\pi^{A'}\wedge \overline{\d\pi^{B'}}-i\Lambda\e_{AB}\pi_{A'}(\sigma\pi)_{B'} \theta^a\wedge\theta^b.
\end{equation}
\end{prop}
\begin{proof}
The local coframe $(\theta^a, \d\pi^{A'}, \overline{\d\pi^{A'}})$ on $\mathbb{S}'\setminus\{\bo\}$ is adapted to the decomposition
\[
\mathbb{C}T(\mathbb{S}'\setminus\{\bo\})=\mathbb{C}\wt{TX}\oplus \mathbb{C}T({\rm fiber}).
\] 
Hence, by the definition of $\mathbb{J}$, we have
\[
 \d\pi^{A'}\circ\mathbb{J}=\sigma^{A'}{}_{\!\!\overline{B'}}\overline{\d\pi^{B'}}, \quad 
 \theta^a\circ\mathbb{J}=J_{B'}{}^{A'}\theta^{AB'}.
\]
Then, applying the formula \eqref{2-form-J} to \eqref{omega}, we compute as
\begin{align*}
\wt\omega_{\mathbb{J}}&=-i \e_{A'B'}\d\pi^{A'}\wedge \sigma^{B'}{}_{\!\!\overline{C'}}\overline{\d\pi^{C'}}-i\Lambda \e_{AB}\pi_{A'}\pi_{B'}J_{C'}{}^{B'}\theta^a\wedge\theta^{BC'} \\
&=i \sigma_{A'\overline{C'}}\d\pi^{A'}\wedge \overline{\d\pi^{C'}}-i\Lambda \e_{AB}\pi_{A'}(\sigma\pi)_{C'}\theta^a\wedge\theta^{BC'}.
\end{align*}
Thus we obtain \eqref{omega-J}.
\end{proof}

\begin{theorem}\label{thm-potential}
We have
\begin{equation}\label{potential}
i\partial\overline{\partial}\|\pi\|^2=\wt\omega_{\mathbb{J}},
\end{equation}
where $\partial\overline{\partial}$ is with respect to $\mathbb{I}$.
In particular, $\wt\omega_{\mathbb{J}}$ is a closed real $(1, 1)$-form.
\end{theorem}
\begin{proof} 
We define $d^c:=(i/2)(\overline\partial-\partial)$ so that $i\partial\overline\partial=dd^c$. For a 1-form $\varphi$, we have $d^c\varphi=-(1/2)d\varphi\circ\mathbb{I}$. 

We fix a point $x\in X$ and compute with local frames for $\mathbb{S}, \mathbb{S}'$ as in the proof of Theorem \ref{I-integrable}. Then we have
\begin{equation}\label{d}
\begin{aligned}
d\|\pi\|^2
&=\e_{A'B'}(\sigma\pi)^{A'}\d\pi^{B'}+
\e_{A'B'}\pi^{B'}\sigma^{A'}{}_{\!\!\overline{C'}}\overline{\d\pi^{C'}} \\
&=(\sigma\pi)_{A'}\d\pi^{A'}+\sigma_{A'\overline{B'}}\pi^{A'}\overline{\d\pi^{B'}}
\end{aligned}
\end{equation}
and hence
\begin{equation}\label{d-c}
-2d^c\|\pi\|^2=i(\sigma\pi)_{A'}\d\pi^{A'}-i\sigma_{A'\overline{B'}}\pi^{A'}\overline{\d\pi^{B'}}.
\end{equation}
We differentiate both sides and use \eqref{d-delta}, \eqref{d-sigma-delta} to obtain
\begin{align*}
-2dd^c\|\pi\|^2&=
-2i\sigma_{A'\overline{B'}}\d\pi^{A'}\wedge \overline{\d\pi^{B'}} 
+ i(\sigma\pi)_{A'}\Lambda\e_{EF}\d_{(E'}{}^{A'}\pi_{F')}\theta^e\wedge\theta^f \\
&\quad -i\pi^{A'}\Lambda\e_{EF}\e_{E'A'}(\sigma\pi)_{F'}\theta^e\wedge\theta^f \\
&= -2i\sigma_{A'\overline{B'}}\d\pi^{A'}\wedge \overline{\d\pi^{B'}}+2i\Lambda\e_{EF}\pi_{E'}(\sigma\pi)_{F'} \theta^e\wedge\theta^f \\
&=-2\wt\omega_{\mathbb{J}}.
\end{align*}
Thus, we have \eqref{potential}.
\end{proof}
We write $\wt\omega_\mathbb{J}=i\wt g_{\a\overline\b}\tilde{\theta}^\a\wedge\overline{\tilde{\theta}^\b}$ with a $(1, 0)$-coframe $(\tilde{\theta}^\a)$ for $\mathbb{I}$, and define a hermitian form $\wt g$ on $\mathbb{S}'\setminus\{\bo\}$ by 
\[
\wt g:=2\wt g_{\a\overline\b}\tilde{\theta}^\a\cdot \overline{\tilde{\theta}^\b}.
\]
Then, since $\d\pi^{A'}, \pi_{A'}\theta^{AA'}$ are $(1, 0)$-forms and $\overline{\d\pi^{A'}}, (\sigma\pi)_{B'}\theta^{BB'}$ are $(0, 1)$-forms, we have
\[
\wt g=2\sigma_{A'\overline{B'}}\d\pi^{A'}\cdot\overline{\d\pi^{B'}}-2\Lambda\e_{AB}\pi_{A'}(\sigma\pi)_{B'} \theta^a\cdot\theta^b.
\]
Noting that $\mathbb{S}'$ is a rank-2 vector bundle, we can compute as
\begin{align*}
2\Lambda\e_{AB}\pi_{A'}(\sigma\pi)_{B'} \theta^a\cdot\theta^b 
&=2\Lambda\e_{AB}\pi_{[A'}(\sigma\pi)_{B']} \theta^a\cdot\theta^b \\
&=\Lambda\e_{AB}\e_{A'B'}\e^{E'F'}\pi_{[E'}(\sigma\pi)_{F']} \theta^a\cdot\theta^b \\
&=-\Lambda\|\pi\|^2 g_{ab}\theta^a\cdot\theta^b.
\end{align*}
Thus, we have
\begin{equation}\label{wt-g}
\wt g=2\sigma_{A'\overline{B'}}\d\pi^{A'}\cdot\overline{\d\pi^{B'}}+\Lambda\|\pi\|^2 g_{ab}\theta^a\cdot\theta^b.
\end{equation}
In particular, $\wt g$ is a positive definite K\"ahler metric when $\Lambda>0$ and a neutral K\"ahler metric when $\Lambda<0$.

\begin{theorem}\label{J-parallel}
Let $\wt\nabla$ be the Levi-Civita connection of $\wt g$. Then we have
\[
\wt\nabla \mathbb{J}=0, \quad \wt\nabla \wt\omega=0.
\]
In particular, $\wt g$ is a hyperk\"ahler metric with respect to the complex structures $\mathbb{I}, \mathbb{J}, \mathbb{K}:=\mathbb{I}\mathbb{J}$.
\end{theorem}
\begin{proof}
We take a local $(1, 0)$-coframe $(\tilde{\theta}^\a)$ on $\mathbb{S}'\setminus\{\bo\}$ with respect to $\mathbb{I}$, and write
\[
\wt\omega_{\mathbb{J}}=i\wt g_{\a\overline\b}\tilde{\theta}^\a\wedge\overline{\tilde{\theta}^\b}, \quad \wt\omega=\wt\omega_{\a\b}\tilde{\theta}^\a\wedge\tilde{\theta}^\b.
\]
where $\wt\omega_{\a\b}=-\wt\omega_{\b\a}$.  Since $\mathbb{J}$ anti-commutes with $\mathbb{I}$, it has components $\mathbb{J}_{\overline\a}{}^\b, \mathbb{J}_{\a}{}^{\overline\b}(=\overline{\mathbb{J}_{\overline\a}{}^\b})$. Then, by \eqref{2-form-J} we have
\[
\wt\omega_{\mathbb{J}}=-i\wt\omega_{\a\b}\tilde{\theta}^\a\wedge(\tilde{\theta}^\b\circ\mathbb{J})
=-i\wt\omega_{\a\b}\mathbb{J}_{\overline\g}{}^\b\tilde{\theta}^\a\wedge\overline{\tilde{\theta}^\g},
\]
which implies
\begin{equation}\label{g-omega}
\wt g_{\a\overline\g}=-\wt\omega_{\a\b}\mathbb{J}_{\overline\g}{}^\b.
\end{equation}
In particular, the skew form $\wt\omega_{\a\b}$ is non-degenerate.

Since $\wt\omega$ is holomorphic, we have $\wt\nabla_{\overline\mu}\wt\omega_{\a\b}=0$ and hence
\[
0=\wt\nabla_{\overline\mu}\wt g_{\a\overline\g}=-\wt\omega_{\a\b}\wt\nabla_{\overline\mu}\mathbb{J}_{\overline\g}{}^\b,
\]
which yields $\wt\nabla_{\overline\mu}\mathbb{J}_{\overline\g}{}^\b=0$.  Contracting both sides of \eqref{g-omega} with $\mathbb{J}_\nu{}^{\overline \g}$ gives
\[
\wt g_{\a\overline\g}\mathbb{J}_\nu{}^{\overline \g}=\wt\omega_{\a\nu}.
\]
Applying $\wt\nabla_{\overline\mu}$, we obtain $\wt\nabla_{\overline\mu}\mathbb{J}_\nu{}^{\overline \g}=0$, or equivalently $\wt\nabla_{\mu}\mathbb{J}_{\overline\nu}{}^{\g}=0$. Thus we obtain 
$\wt\nabla\mathbb{J}=0$, and by \eqref{g-omega} we also have $\wt\nabla\wt\omega=0$.
\end{proof}
Since $\wt\omega$ is a holomorphic symplectic form on $\mathbb{S}'\setminus\{\bo\}$, the 1-form 
$\wt\tau$ defines a (weighted) holomorphic contact form on $\mathbb{P}(\mathbb{S}')$, and via the 3-form $\wt\tau\wedge d\wt\tau$, we can identify $\mathbb{S}'\setminus\{\bo\}$ with a fractional power of the canonical bundle of $\mathbb{P}(\mathbb{S}')$ with the zero section removed:

\begin{prop}\label{canonical-bundle}
The $3$-form $\wt\tau\wedge d\wt\tau$ descends to a section of the 
$\mathbb{C}^*$-bundle
\[
(\mathbb{S}'\setminus\{\bo\})^{\otimes(-4)}\otimes (K_{\mathbb{P}(\mathbb{S}')}\setminus\{\bo\})
\]
over $\mathbb{P}(\mathbb{S}')$ and defines an isomorphism 
\[
\mathbb{S}'\setminus\{\bo\}\cong K_{\mathbb{P}(\mathbb{S}')}^{1/4}\setminus\{\bo\}.
\]
\end{prop}
\begin{proof}
Let $E:=\pi^{A'}(\partial/\partial \pi^{A'})$ be the $(1, 0)$-part of the Euler field on $\mathbb{S}'\setminus\{\bo\}$. Then, since
\[
\wt\tau(E)=\e_{A'B'}\pi^{A'}\pi^{B'}=0,
\]
$\wt\tau$ and $\wt\tau\wedge d\wt\tau$ are horizontal forms with respect to the fibration $\mathbb{S}'\setminus\{\bo\}\to \mathbb{P}(\mathbb{S}')$ which are homogeneous of degree $(2, 0)$ and $(4, 0)$ respectively. Moreover, as is shown in the proof of Theorem \ref{J-parallel}, $\wt\omega$ is a non-degenerate 2-form and hence  
\[
0\neq E\lrcorner\, \wt\omega^2=2(E\lrcorner\, d\wt\tau)\wedge d\wt\tau
=2(\mathcal{L}_E \wt\tau)\wedge d\wt\tau =4\wt\tau\wedge d\wt\tau,
\]
where $\mathcal{L}_E$ denotes the Lie derivative. Thus, $\wt\tau\wedge d\wt\tau$ defines a section of $(\mathbb{S}'\setminus\{\bo\})^{\otimes (-4)} \otimes (K_{\mathbb{P}(\mathbb{S}')}\setminus\{\bo\})$.
\end{proof}

\section{The K\"ahler-Einstein metric on $\mathbb{P}(\mathbb{S}')$}\label{MA-Kahler-Einstein}

The construction of hyperk\"ahler metrics on $\mathbb{S}'\setminus\{\bo\}$ in \S\ref{hyperkahler-spinor} also gives rise to a K\"ahler-Einstein metric on the twistor space $\mathbb{P}(\mathbb{S}')$. To see this, we first recall a general relationship between the Ricci-flat and the Einstein equations via the complex Monge--Amp\`ere equation.

\subsection{The Ricci-flat equation and the complex Monge--Amp\`ere equation}\label{Ricci-flat}
Let $Y$ be an $m$-dimensional complex manifold, and $\mathcal{L}\to Y$ a holomorphic line bundle.
We consider a (pseudo-)K\"ahler metric $\wt g$ on $\mathcal{L}\setminus\{\bo\}$ with the K\"ahler form
\[
\omega_{\wt g}=i\partial\overline\partial\varphi,
\]
where $\varphi$ is a real function on $\mathcal{L}\setminus\{\bo\}$. We assume that $\varphi$ is homegeneous of degree $(1, 1)$, i.e., $\varphi=|z^0|^2\underline\varphi$ locally for a fiber coordinate $z^0$ and a real function $\underline{\varphi}$ on $Y$. Then, $\wt g$ is Ricci-flat if and only if the complex Monge--Amp\`ere equation
\[
\partial\overline\partial\log |\omega_{\wt g}^{m+1}|\bigl(=\partial\overline\partial\log|\det\wt g\,|\,\bigr)=0
\]
is satisfied. Let us assume that $\varphi$ is non-vanishing and consider 
\[
\omega_{\rm KE}\!:=i\partial\overline\partial\log|\varphi|,
\]
which descends to a $(1, 1)$-form on $Y$. Then we have 
\begin{equation}\label{MA-descend}
\omega_{\wt g}^{m+1}=(m+1)\varphi^{\,m+1}\Bigl(i\frac{dz^0}{z^0}\wedge\frac{d\overline{z^0}}{\overline{z^0}}\Bigr)\wedge \omega_{\rm KE}^m,
\end{equation}
which implies that $\omega_{\rm KE}$ is non-degenerate and defines a (pseudo-)K\"ahler metric $g_{\rm KE}$ on $Y$. Moreover, taking $i\partial\overline\partial\log|\cdot|$ of both sides of \eqref{MA-descend} gives
\[
i\partial\overline\partial\log |\omega_{\wt g}^{m+1}|=i\partial\overline\partial\log |\omega_{\rm KE}^{m}|+(m+1)\omega_{\rm KE}
\]
and hence
\[
{\rm Ric}(\wt g)={\rm Ric} (g_{\rm KE})+(m+1)g_{\rm KE}.
\]
Thus, if $\wt g$ is Ricci-flat, $g_{\rm KE}$ becomes a K\"ahler-Einstein metric on $Y$.

\subsection{The structure of the K\"ahler-Einstein metric on $\mathbb{P}(\mathbb{S}')$}
Now we apply the construction to the hyperk\"ahler metric $\wt g$ on the $\mathbb{C}^*$-bundle 
$\mathbb{S}'\setminus\{\bo\}\to \mathbb{P}(\mathbb{S}')$. Since any hyperk\"ahler metric is Ricci-flat, the K\"ahler potential $\|\pi\|^2$ solves the complex Monge--Amp\`ere equation and hence
\[
\omega_{\rm KE}:=i\partial\overline{\partial}\log\|\pi\|^2
\]
descends to the K\"ahler form of a K\"ahler-Einstein metric $g_{\rm KE}$ on $\mathbb{P}(\mathbb{S}')$. The following theorem describes the structure of this metric:

\begin{theorem}\label{Kahler-Einstein}
Let $H\subset T\mathbb{P}(\mathbb{S}')$ be the distribution obtained by the horizontal lift of $TX$ with respect to $\nabla$. Then, in terms of the decomposition
\[
T\mathbb{P}(\mathbb{S}')=H\oplus T({\rm fiber}),
\]
the K\"ahler-Einstein metric $g_{\rm KE}$ is given by
\[
g_{\rm KE}=(\Lambda g)\oplus g_{\rm FS},
\]
where $g_{\rm FS}$ denotes the Fubini--Study metric on each fiber $\mathbb{P}(\mathbb{S}'_x)\cong \mathbb{P}^1$. In particular, $g_{\rm KE}$ is positive definite when $\Lambda>0$, and Lorentzian when $\Lambda<0$.
\end{theorem}
\begin{proof}
We take a local frame for $\mathbb{S}'$ such that
\begin{equation}\label{potential-pi}
\|\pi\|^2=|\pi^{0'}|^2+|\pi^{1'}|^2.
\end{equation}
Since fibers of $\mathbb{S}'\setminus\{\bo\}$ are complex submanifolds (with respect to $\mathbb{I}$), the operator $\partial\overline\partial$ commutes with the restriction of the K\"ahler potential to the fibers. Hence the restriction of $\omega_{\rm KE}$ to each fiber agrees with the K\"ahler form of the Fubini--Study metric. (Note that the fiber coordinates $(\pi^{0'}, \pi^{1'})$ are not holomorphic on the total space $\mathbb{S}'\setminus\{\bo\}$, so the expression \eqref{potential-pi} is not useful for the computation except along the fibers.)

By using \eqref{omega-J}, \eqref{d} and \eqref{d-c}, we can compute $\omega_{\rm KE}$ as follows:
\begin{align*}
dd^c\log\|\pi\|^2
&=\frac{\wt\omega_{\mathbb{J}}}{\|\pi\|^2}+\frac{d\|\pi\|^2\wedge d^c\|\pi\|^2}{\|\pi\|^4} \\
&=i\|\pi\|^{-2}\bigl(\sigma_{A'\overline{B'}}\d\pi^{A'}\wedge \overline{\d\pi^{B'}}-\Lambda\e_{AB}\pi_{A'}(\sigma\pi)_{B'} \theta^a\wedge\theta^b\bigr) \\
&\quad +i\|\pi\|^{-4}(\sigma\pi)_{A'}\d\pi^{A'}\wedge\overline{(\sigma\pi)_{B'}\d\pi^{B'}}.
\end{align*}
From this, we see that $\omega_{\rm KE}(V, W)=0$ for any $V\in H$ and $W\in T({\rm fiber})$. Since $H$ and $T({\rm fiber})$ are complex subbundles of $T\mathbb{P}(\mathbb{S}')$, this implies that these are orthogonal to each other with respect to $g_{\rm KE}$. 

The restriction of $\omega_{\rm KE}$ to $H$ is given by
\[
\omega_{\rm KE}|_H=
-i\|\pi\|^{-2}\Lambda\e_{AB}\pi_{A'}(\sigma\pi)_{B'} \theta^a\wedge\theta^b.
\]
Since $\pi_{A'}\theta^a$ is a $(1, 0)$-form and 
$(\sigma\pi)_{B'} \theta^b$ is a $(0, 1)$-form with respect to $\mathbb{I}$, we have
\begin{align*}
g_{\rm KE}|_H=-2\|\pi\|^{-2}\Lambda\e_{AB}\pi_{A'}(\sigma\pi)_{B'}
\theta^a\cdot\theta^b=-2\|\pi\|^{-2}\Lambda\e_{AB}\pi_{[A'}(\sigma\pi)_{B']}
\theta^a\cdot\theta^b.
\end{align*}
Noting that ${\rm rank}\,\mathbb{S'}=2$, we have
\[
\pi_{[A'}(\sigma\pi)_{B']}=\frac{1}{2}\e^{P'Q'}\pi_{P'}(\sigma\pi)_{Q'}\e_{A'B'}=
-\frac{1}{2}\|\pi\|^2\e_{A'B'}.
\]
Thus we obtain
\[
g_{\rm KE}|_H=\Lambda \e_{AB}\e_{A'B'}\theta^a\cdot\theta^b=\Lambda g
\]
and complete the proof.
\end{proof}


\section{Ambient metrics in conformal and CR geometries}\label{ambient-conf-CR}

\subsection{The Fefferman--Graham ambient metric}\label{ambient}

Let us review the notion of ambient metric in conformal geometry, introduced by Fefferman--Graham \cite{FG1}; we refer the reader to the book \cite{FG2} for the  detail of the theory.

Let $(M, [g])$ be an $n$-dimensional $C^\infty$ conformal manifold of signature $(p, q)$.
The conformal structure defines an $\mathbb{R}_+$-subbundle $\mathcal{G}\subset S^2(T^*M)$ with the tautological 2-tensor $\boldsymbol{g}_0$. We consider the $(n+2)$-dimensional manifold
\[
\widetilde{\mathcal{G}}:=\mathcal{G}\times (-\varepsilon, \varepsilon)_\rho,
\]
called the {\it ambient space}, and identify $\mathcal{G}$ with $\mathcal{G}\times\{0\}$.
Let $\d_s\colon \widetilde{\mathcal{G}}\to \widetilde{\mathcal{G}}\ (s\in\mathbb{R}_+)$ be the dilation of the $\mathcal{G}$-component. Fefferman--Graham \cite{FG1, FG2} constructed a pseudo-Riemannian metric $\wt g$ of signature $(p+1, q+1)$ on $\widetilde{\mathcal{G}}$, called the {\it ambient metric}, which satisfies the homogeneity condition $\d_s^*\wt g=s^2\wt g$, the boundary condition $\wt g|_{T\mathcal{G}}=\boldsymbol{g}_0$, and the asymptotic Ricci-flat equation
\[
{\rm Ric}(\widetilde g)=\begin{cases}
O(\rho^\infty) & (n: {\rm odd}) \\
O^+(\rho^{n/2-1}) & (n:  {\rm even}).
\end{cases}
\]
Here, for a 2-tensor $S$ on $\widetilde{\mathcal{G}}$, the equation $S=O^+(\rho^m)$ means that $S=O(\rho^m)$ and $(\rho^{-m}S)|_{T\mathcal{G}}$ is a horizontal 2-tensor with vanishing trace with respect to $\boldsymbol{g}_0$. The ambient metric is unique in the sense that if $\wt g{\,'}$ satisfies the same conditions, there exists a dilation-equivariant differeomphism 
$\phi\colon \widetilde{\mathcal{G}}\to \widetilde{\mathcal{G}}$ which satisfies $\phi|_{\mathcal{G}}={\rm id}$  and 
\[
\phi^*{\widetilde g}{\,'}-\widetilde g =\begin{cases}
O(\rho^\infty) &(n: {\rm odd}) \\
O^+(\rho^{n/2}) & (n: {\rm even}) .
\end{cases}
\]
When $n$ is even, the equation ${\rm Ric}(\wt g)=O^+(\rho^{n/2-1})$ is optimal in general since 
the 2-tensor $\mathcal{O}=(\rho^{1-n/2}{\rm Ric}(\wt g))|_{T\mathcal{G}}$, called the {\it obstruction tensor}, is independent of the choice of $\wt g$, and is non-zero for general even dimensional conformal manifolds. 

The flat model of conformal manifold of signature $(p, q)$ is the hyperquadric
\[
\bigl\{[\xi]\in \mathbb{RP}^{n+1}| -(\xi^0)^2-(\xi^1)^2-\cdots-(\xi^q)^2+(\xi^{q+1})^2+\cdots+(\xi^{n+1})^2=0\bigr\}.
\]
In this case, the ambient space can be identified with a neighborhood of the cone in $\mathbb{R}^{n+2}$, and the ambient metric is given by the flat metric
\[
\wt g= -(d\xi^0)^2-(d\xi^1)^2-\cdots-(d\xi^q)^2+(d\xi^{q+1})^2+\cdots+(d\xi^{n+1})^2.
\]

\subsection{The Poincar\'e-Einstein metric}

There is another class of natural metrics associated to conformal manifolds.
Let $M$ be a hypersurface in an $(n+1)$-dimensional manifold $X$ and $r$ a defining function of $M$, i.e., $M=\{r=0\}$ and $dr|_M\neq0$. 
A (pseudo-)Riemannian metric $g_+$ on $X\setminus M$ is said to be {\it conformally compact} if $\overline{g}_+:=r^2g_+$ can be extended to a metric on $X$. If this holds, the conformal class 
$[\overline{g}_+|_{TM}]$ is independent of the choice of $r$, and is called the two-sided) {\it conformal infinity} of $g_+$. If, further, $g_+$ satisfies the Einstein equation,
it is called a {\it Poincar\'e-Einstein metric}. The flat model of the Poincar\'e-Einstein metric is the hyperbolic metric $4|dx|^2/(1-|x|^2)^2$, whose conformal infinity is the standard conformal sphere $S^n$. The construction of the Poincar\'e-Einstein metric with a given conformal infinity is a major problem in conformal geometry, and there are some known existence results (e.g., \cite{GL, GS, HP, LeB1, Mat}). Fefferman--Graham \cite{FG2} also shows how to translate the construction of the ambient metric to that of the Poincar\'e-Einstein metric.

In our construction of hyperk\"ahler ambient metrics, an important role is played by the following existence theorem in the real analytic case, which was proved by LeBrun \cite{LeB1} by a twistorial method:

\begin{theorem}[\cite{LeB1}]\label{LeBrun}
Any $3$-dimensional real analytic (positive definite) conformal manifold $(\Sigma, [h])$ is the conformal infinity of a $4$-dimensional Poincar\'e-Einstein manifold $(X, g_+)$ which satisfies the anti-self dual equation $W^+=0$.
\end{theorem}


\subsection{Special defining functions}
For each choice of a representative metric $g$ on the conformal infinity $M$, one can construct a unique defining function $r$ of $M\subset X$ defined near $M$ which satisfies
\[
\overline{g}_+|_{TM}=g, \quad |dr|^2_{\overline{g}_+}=-2\Lambda,
\] 
where $\overline{g}_+:=r^2 g_+$ and $\Lambda=R/2n(n+1)<0$ is a multiple of the scalar curvature of $g_+$, which is necessarily a negative constant. This is called the {\it special defining function}
associated to $g$ (\cite{G}). The following proposition will be used in 
 the proof of Theorem \ref{CR-isom}.
 
 \begin{prop}\label{totally}
 Let $r$ be the special defining function associate to a representative metric $g$ on $M$. Then, $(M, g)$ is a totally geodesic submanifold of 
 $(X, \overline{g}_+)$. 
 \end{prop}
 
\begin{proof}
Since $g_+$ is an Einstein metric, its Schouten tensor 
\[
 P_{ab}:=\frac{1}{n-1}\Bigl(R_{ab}-\frac{R}{2n}(g_+)_{ab}\Bigr)
\]
 satisfies $P_{ab}=\Lambda (g_+)_{ab}$. Let $\overline P_{ab}$ be the Schouten tensor of the compactified metric $\overline g_+$ and $\overline\nabla$ the Levi-Civita connection of $\overline g_+$.
  Since $g_+$ and $\overline g_+$ are related by $g_+=e^{2\U}\overline g_+$ with $\U=-\log|r|$, the conformal rescaling formula for the Schouten tensor gives
\begin{align*}
P_{ab}&=\overline{P}_{ab}-\overline\nabla _a\U_b+\U_a\U_b-\frac{1}{2}\U_c\U^c(\overline{g}_+)_{ab} \\
&=\overline{P}_{ab}+\frac{\overline\nabla_a r_b}{r}-\frac{1}{2r^2}|dr|^2_{\overline g_+}(\overline g_+)_{ab} \\
&=\overline{P}_{ab}+\frac{\overline\nabla_a r_b}{r}+\frac{\Lambda}{r^2}
(\overline g_+)_{ab}.
\end{align*}
 It follows that $\overline\nabla_a r_b=-r\overline P_{ab}$ and hence $M$ is totally geodesic.
 \end{proof}


\subsection{CR manifolds}

Let $M$ be a $(2n+1)$-dimensional $C^\infty$-manifold. 
A rank $n$ complex sub-bundle $T^{1, 0}M\subset\mathbb{C}TM$ is called a {\it CR structure} if it satisfies
\[
T^{1, 0}M\cap T^{0, 1}M=\{0\}, \quad 
[\Gamma(T^{1, 0}M), \Gamma(T^{1, 0}M)]\subset \Gamma(T^{1, 0}M),
\]
where we set $T^{0, 1}M:=\overline{T^{1, 0}M}$.
When $M$ is a real hypersurface in a complex manifold $Y$, it is endowed with the natural CR structure $T^{1, 0}M:=(\mathbb{C}TM)\cap T^{1, 0}Y$. An embedding $f\colon M\to Y$ from a CR manifold $M$ to a complex manifold $Y$ is called a {\it CR embedding} when $f_*(T^{1, 0}M)\subset T^{1, 0}Y$.

When $T^{1, 0}M$ is a CR structure, its real part $H:={\rm Re}\,T^{1, 0}M\subset TM $ becomes a rank $2n$ real sub-bundle. The {\it Levi form} is the $TM/H$-valued hermitian form on $T^{1, 0}M$ defined by
\[
h(Z, W):=i[Z, \overline{W}]\  {\rm mod}\  H.
\]
We say a CR structure is {\it non-degenerate} when the Levi form is non-degenerate, and {\it strictly pseudoconvex} when it is positive definite for a trivialization of $TM/H$.


\subsection{The ambient metric associated with embedded CR manifolds}
To any non-degenerate CR manifold $M$ whose Levi form has signature $(p, q)$, one can associate a conformal metric of signature $(2p+1, 2q+1)$ called the {\it Fefferman metric} on a circle bundle over $M$ (\cite{F, BDS, Lee}). Then, we can define the ambient metric for the CR manifold $M$ as that of the associated Fefferman metric. However, when $M$ is realized as a real hypersurface in a complex manifold, there is a more direct way to define the notion of ambient metric, introduced by Fefferman \cite{F}. The construction is via an (approximate) solution to the complex Monge--Amp\`ere equation in \S\ref{Ricci-flat} and described as follows.

Let $(M, T^{1, 0}M)$ be a $(2n+1)$-dimensional non-degenerate CR manifold. We assume that $M$ is embedded as a real hypersurface in an $(n+1)$-dimensional complex manifold $Y$.
Let $\mathcal{L}\to Y$ be a holomorphic line bundle, and $\wt r$ a defining function of 
$\mathcal{L}|_M\subset \mathcal{L}$ homogeneous of degree $(1, 1)$, i.e., one can locally write $\wt r=|z^0|^2r$ with a fiber coordinate $z^0$ and a defining function $r$ of $M\subset Y$. 
If the Levi form of $M$ has signature $(p, q)$, the $(1, 1)$-form
\[
\omega_{\wt g}:=i\partial\overline{\partial} \wt r
\]
on $\wt{\mathcal{G}}:=\mathcal{L}\setminus\{\bo\}$ defines a pseudo-K\"ahler metric $\wt g$ of signature $(p+1, q+1)$ in a neighborhood of $\wt{\mathcal{G}}|_M$. We say $\wt g$ is an {\it ambient metric} for the CR manifold $M$ if it satisfies the approximate complex Monge--Amp\`ere equation 
\begin{equation}\label{approximate-MA}
({\rm Ric}(\wt g)=)\partial\overline\partial\log |\omega_{\widetilde g}^{n+2}|=\partial\overline\partial\cdot O(\rho^{\,n+2}),
\end{equation}
where $\rho$ is an arbitrary defining function of $\mathcal{L}|_M\subset \mathcal{L}$.
In the case where $Y=\mathbb{C}^{n+1}$ and $\mathcal{L}=K^{1/(n+2)}_{\mathbb{C}^{n+1}}$,
Fefferman \cite{F} constructed such an $\wt r$, which is unique modulo $O(\rho^{\,n+3})$. 
The order $O(\rho^{\,n+2})$ in the right-hand side of \eqref{approximate-MA} is optimal in general if we consider only smooth defining functions. 

Restricting $\wt g$ to $\wt{\mathcal{G}}|_M$, we obtain a conformal metric $[g_{\rm F}]$, called the {\it Fefferman metric}, on the circle bundle $\mathcal{C}=\wt{\mathcal{G}}|_M/\mathbb{R}_+$ over $M$, and $\wt g$ can be interpreted as the Fefferman--Graham ambient metric for $(\mathcal{C}, [g_{\rm F}])$.

In the case of the hyperquadric 
\[
M=\bigl\{[\xi]\in \mathbb{CP}^{n+1}\ \big|\ \wt r:= -|\xi^0|^2-|\xi^1|^2-\cdots- 
|\xi^q|^2+|\xi^{q+1}|^2+\cdots+|\xi^{n+1}|^2=0 \bigr\},
\]
which is the flat model of non-degenerate CR manifold, the $(1, 1)$-form $(i/2)\partial\overline\partial \wt r$ gives the flat ambient metric 
\[
\wt g= -|d\xi^0|^2-|d\xi^1|^2-\cdots- 
|d\xi^q|^2+|d\xi^{q+1}|^2+\cdots+|d\xi^{n+1}|^2
\] 
on $\wt{\mathcal{G}}=\mathcal{O}_{\mathbb{P}^{n+1}}(-1)\setminus\{\bo\}=\mathbb{C}^{n+2}\setminus\{0\}$.


\subsection{The smooth Cheng--Yau metric}
By the computation in \S\ref{Ricci-flat}, the $(1, 1)$-form 
\[
\omega_{\rm CY}:=-i\partial\overline{\partial}\log|\wt r|
\]
descends to the K\"ahler form of a (pseudo-)K\"ahler metric $g_{\rm CY}$ on $Y\setminus M$ in a neighborhood of $M$ which satisfies the approximate Einstein equation
\[
{\rm Ric} (g_{\rm CY})+(n+2)g_{\rm CY}=\partial\overline\partial\cdot O(\rho^{\,n+2}).
\]
The K\"ahler metric $g_{\rm CY}$ has signature $(p+1, q)$ on $\{r<0\}$ and $(q+1, p)$ on $\{r>0\}$.
When $M$ is the boundary of a bounded strictly pseudoconvex domain $\Omega\subset\mathbb{C}^{n+1}$, Cheng and Yau \cite{ChY} constructed 
a complete K\"ahler-Einstein metric on $\Omega$, called the {\it Cheng--Yau metric}, via the exact exact solution to the complex Monge--Amp\`ere equation which has a logarithmic singularity at the boundary. In this case, the metric $g_{\rm CY}$ provides a smooth approximation to the Cheng--Yau metric; accordingly, even in more general settings, 
we refer to $g_{\rm CY}$ as the {\it smooth Cheng--Yau metric}.


\section{Hyperk\"ahler ambient metrics}\label{hyperkahler-ambient}

\subsection{Twistor CR manifolds}
We will review the definition of twistor CR manifolds introduced by LeBrun \cite{LeB2}.

Let $(\Sigma, [h])$ be a 3-dimensional $C^\infty$ conformal manifold. We extend $h\in[h]$ to a complex bi-linear form on $\mathbb{C}T^*\Sigma$ and set 
\[
\wh M:=\{\zeta\in \mathbb{C}T^*\Sigma\ |\ h(\zeta, \zeta)=0, \zeta\neq0\}
\]
and 
\[
M:=\wh M/{\mathbb{C}^*}\subset \mathbb{P}(\mathbb{C}T^*\Sigma).
\]
Then, $M$ is a 5-dimensional manifold which is a $\mathbb{P}^1$-bundle over $\Sigma$. Let 
\[
\omega:=d(\zeta_i dx^i)=d\zeta_i\wedge dx^i
\]
be the canonical 2-form on $\mathbb{C}T^*\Sigma$, where $(x^i)$ are local coordinates of $\Sigma$ and $(\zeta_i)$ are associated fiber coordinates for
 $\mathbb{C}T^*\Sigma$. Since the distribution
 \[
 D:=\{V\in \mathbb{C}T\wh M\ |\ (V\lrcorner\, \omega)|_{\mathbb{C}T\wh M}=0\}
 \]
is $\mathbb{C}^*$-invariant, we have a rank-2 sub-bundle
\[
T^{0, 1}M:=\mathring\pi_* D\subset \mathbb{C}TM,
\]
where $\mathring\pi\colon \wh M\to M$ is the projection. It can be shown that $T^{1, 0}M:=\overline{T^{0, 1}M}$ defines a non-degenerate CR structure whose Levi form has Lorentzian (neutral) signature $(1, 1)$. The CR manifold $(M, T^{1, 0}M)$ is determined by the conformal structure $[h]$ and is called the {\it twistor CR manifold} associated to $(\Sigma, [h])$. 

The CR structure is described in terms of local coordinates as follows:
\begin{prop}\label{twistor-CR-local}
Fix a representative metric $h\in[h]$, and let $(x^i, \zeta_i)$ be local coordinates of 
$\mathbb{C}T^*\Sigma$ as above. Then we have
\begin{equation}\label{D}
D=\Bigl\{c\,\zeta^i X_i+w_i\frac{\partial}{\partial \overline{\zeta}_i}\ \Big|\ c\in\mathbb{C},\ \overline{\zeta}_iw^i=0\Bigr\}, \quad 
\mathring\pi^{*}(T^{0, 1}M)\cong D\Big/\mathbb{C}\,\overline\zeta_i\frac{\partial}{\partial \overline\zeta_i},
\end{equation}
where indices are raised or lowered by $h$, and $X_i$ denotes the horizontal lift of $\partial/\partial x^i$ by the Levi-Civita connection of $h$.
\end{prop}
\begin{proof}
Since these expressions are independent of the choice of coordinates, we can take $(x^i)$ to be normal coordinates of $h$ around a fixed point $p\in\Sigma$. A complex vector 
\[
V=u^i\frac{\partial}{\partial x^i}+v_i\frac{\partial}{\partial \zeta_i}+w_i\frac{\partial}{\partial \overline\zeta_i}
\]
satisfies $V\in \mathbb{C}T\wh M$ if and only if $V\cdot h(\zeta, \zeta)=V\cdot \overline{h(\zeta, \zeta)}=0$, and this condition can be written as
\[
v^i\zeta_i=w^i\overline\zeta_i=0
\]
on the fiber of $p$. Moreover, $V\in D$ if and only if 
\[
V\lrcorner\, \omega=v_idx^i-u^id\zeta_i
\]
is a linear combination of $\zeta^id\zeta_i$ and $\overline{\zeta^i}d\overline\zeta_i$. 
Thus, $v_i=0, u^i=c\,\zeta^i\,(c\in\mathbb{C})$, so we obtain the above expression for $D$. Since ${\rm Ker}\,\mathring\pi_*=\mathbb{C}\zeta_i(\partial/\partial \zeta_i)\oplus
\mathbb{C}\overline\zeta_i(\partial/\partial \overline\zeta_i)$, we have 
$D\cap {\rm Ker}\,\mathring\pi_*=\mathbb{C}\overline\zeta_i(\partial/\partial \overline\zeta_i)$ and obtain the expression for $\mathring\pi^{*}(T^{0, 1}M)$.
\end{proof}


\subsection{The complex structures $\mathbb{I}_r$ on the conformal spinor bundle}

Now we assume that $(\Sigma, [h])$ is real analytic and realize it as the conformal infinity of an anti self-dual Poincar\'e-Einstein manifold $(X, g_+)$ by using Thereom \ref{LeBrun}. 
We let $\mathbb{S}, \mathbb{S}'$ be the conformal spinor bundles over $(X, [\overline{g}_+])$ and set
\[
\wt{\mathcal{G}}:=\mathbb{S}'\setminus\{\bo\}.
\] 
For each choice of a defining function $r$ of $\Sigma\subset X$, we have the representative metric $\overline{g}_+=r^2g_+$ and the skew forms $\overline{\e}_{AB}, \overline{\e}_{A'B'}$
satisfying $(\overline{g}_+)_{ab}=\overline{\e}_{AB}\overline{\e}_{A'B'}$. On $X\setminus\Sigma$, 
we have $(g_+)_{ab}=\e_{AB}\e_{A'B'}$ with
\[
\e_{AB}=\frac{1}{|r|}\overline{\e}_{AB}, \quad \e_{A'B'}=\frac{1}{|r|}\overline{\e}_{A'B'}
\]
by the rescaling formula \eqref{e-rescale}.

Let $\overline\nabla$ be the connection on $\mathbb{S}'$ induced by the Levi-Civita connection of $\overline{g}_+$. The associated horizontal distribution on $\wt{\mathcal{G}}$ is given by the kernel of the 1-forms
\[
d\pi^{A'}+\overline{\Gamma}_{cB'}{}^{A'}\pi^{B'}\theta^c,\quad \overline{d\pi^{A'}}+\overline{\Gamma}_{c\overline{B'}}{}^{\overline{A'}}\overline{\pi^{B'}}\theta^c
\]
where $(\overline{\Gamma}_{cB'}{}^{A'}\theta^c)$ is the connection form of $\overline\nabla$, and $\overline{\Gamma}_{c\overline{B'}}{}^{\overline{A'}}\theta^c:=\overline{\overline{\Gamma}_{cB'}{}^{A'}\theta^c}$. The horizontal lift of a vector $v=v^a(\partial/\partial x^a)\in\mathbb{C}TX$ to $\mathbb{C}T_\pi\wt{\mathcal{G}}$ is written as
\[
\wt v=v^a\frac{\partial}{\partial x^a}-v^c\Bigl(\overline\Gamma_{cB'}{}^{A'}\pi^{B'}\frac{\partial}{\partial \pi^{A'}}+\overline\Gamma_{c\overline{B'}}{}^{\overline{A'}}\overline{\pi^{B'}}\frac{\partial}{\partial \overline{\pi^{A'}}}\Bigr).
\]

We define an almost complex structure $\mathbb{I}_r$ on $\wt{\mathcal{G}}$ as in \S\ref{def-I}; namely, we define
\begin{equation}\label{T01-I}
{}^{\mathbb{I}_r}T^{0, 1}_\pi \wt{\mathcal{G}}:=\wt{\mathcal{A}}_{[\pi]}\oplus T^{0,1}_\pi({\rm fiber}),
\end{equation}
where $\wt{\mathcal{A}}_{[\pi]}$ is the horizontal lift of the $\a$-plane $\mathcal{A}_{[\pi]}\subset\mathbb{C}TX$ with respect to $\overline\nabla$. We can also define an almost complex structure $\mathbb{I}_1$ on $\wt{\mathcal{G}}|_{X\setminus\Sigma}$ by using the connection $\nabla$ determined by the  metric $g_+$ in place of $\overline\nabla$. Then, $\mathbb{I}_1$ is integrable by Theorem \ref{I-integrable} since $g_+$ is an anti self-dual Einstein metric. 
The following proposition shows that the almost complex structures $\mathbb{I}_r$ are also integrable and related to each other by dilations on $\wt{\mathcal{G}}$:

\begin{prop}\label{dilation}
Let $r, \wh r=e^\U r$ be defining functions of $\Sigma\subset X$. Then the map
\begin{equation}\label{biholo1}
\d_{e^\U}\colon (\wt{\mathcal{G}}, \mathbb{I}_{\wh r})\longrightarrow
(\wt{\mathcal{G}}, \mathbb{I}_{r}), \quad \d_{e^\U}(\pi):=e^\U\pi
\end{equation}
is a biholomorphism. Similarly, the map
\begin{equation}\label{biholo2}
\d_{|r|}\colon (\wt{\mathcal{G}}|_{X\setminus\Sigma}, \mathbb{I}_{r})\longrightarrow
(\wt{\mathcal{G}}|_{X\setminus\Sigma}, \mathbb{I}_{1})
\end{equation}
is a biholomorphism. In particular, the almost complex structure $\mathbb{I}_r$ is integrable.
\end{prop}
\begin{proof}
Let $\overline\nabla$, $\wh\nabla$ be the connections on $\mathbb{S}'$ determined by $r^2g_+$, $\wh r^{\,2}g_+$, and let $\wt v, \wh v$ be the horizontal lifts of a complex vector $v=v^a(\partial/\partial x^a)=\xi^{A}\pi^{A'}(\partial/\partial x^a)\in \mathcal{A}_{[\pi]}$ by $\overline\nabla$, $\wh\nabla$. By the conformal transformation formula \eqref{spinor-conf}, the connection forms are related by
\[
\overline\Gamma_{cA'}{}^{B'}=\wh\Gamma_{cA'}{}^{B'}-\d_{C'}{}^{B'}\U_{CA'}.
\]
Hence we have
\begin{align*}
\overline\Gamma_{cA'}{}^{B'}v^c\pi^{A'}&=\wh\Gamma_{cA'}{}^{B'}v^c\pi^{A'}-\U_{CA'}\xi^C\pi^{A'}\pi^{B'} \\
&=\wh\Gamma_{cA'}{}^{B'}v^c\pi^{A'}-\U_a v^a\pi^{B'}.
\end{align*}
Using this formula and 
\[
(\d_{e^{\U}})_*\frac{\partial}{\partial x^a}=\frac{\partial}{\partial x^a}+e^{\U}\U_a\Bigl(\pi^{B'}\frac{\partial}{\partial \pi^{B'}}+\overline{\pi^{B'}}\frac{\partial}{\partial \overline{\pi^{B'}}}\Bigr), \quad (\d_{e^{\U}})_*\frac{\partial}{\partial \pi^{B'}}=e^{\U}\frac{\partial}{\partial \pi^{B'}},
\]
we have
\begin{align*}
(\d_{e^{\U}})_* (\wh v_\pi)=v^a\frac{\partial}{\partial x^a}-v^c\Bigl(\overline\Gamma_{cB'}{}^{A'}e^\U\pi^{B'}\frac{\partial}{\partial \pi^{A'}}+\overline\Gamma_{c\overline{B'}}{}^{\overline{A'}}e^\U\overline{\pi^{B'}}\frac{\partial}{\partial \overline{\pi^{A'}}}\Bigr)=\wt v_{e^\U\pi}.
\end{align*}
This implies that the horizontal lifts of $\mathcal{A}_{[\pi]}$ by $\wh\nabla$ and $\overline\nabla$ are related by $(\d_{e^\U})_*$. Also,  $\d_{e^\U}$ is biholomorphic on each fiber $(\cong \mathbb{C}^2\setminus\{0\})$ of $\wt{\mathcal{G}}$. Thus, the map \eqref{biholo1} is biholomorphic. 
The case of the map \eqref{biholo2} can be shown in the same way.
\end{proof}

For each $\pi\in\wt{\mathcal{G}}_x$, we define an endomorphism
\[
I_\pi\colon \mathbb{C}T_x X\longrightarrow \mathbb{C}T_x X
\]
by
\[
v^{AA'}\frac{\partial}{\partial x^a}\longmapsto v^{AB'}I_{B'}{}^{A'}\frac{\partial}{\partial x^a}
\]
with 
\[
I_{B'}{}^{A'}:=\frac{-i}{\|\pi\|^2_r}
\bigl(\overline\e_{C'B'}\pi^{C'}(\sigma\pi)^{A'}+\overline\e_{C'B'}(\sigma\pi)^{C'}\pi^{A'}\bigr),
\]
where
\begin{equation}\label{norm-r}
\|\pi\|^2_r:=\overline\e_{C'D'}(\sigma\pi)^{C'}\pi^{D'}\ (>0).
\end{equation}
Note that $I_\pi$ is a real operator which is independent of the choice of $r$ and that we may use $\e_{A'B'}, \|\pi\|^2=\e_{C'D'}(\sigma\pi)^{C'}\pi^{D'}$ instead of $\overline{\e}_{A'B'}, \|\pi\|_r^2$ when we work on $\wt{\mathcal{G}}|_{X\setminus\Sigma}$.
As we showed in \S\ref{def-I}, in terms of the decomposition 
\begin{equation}\label{decom-G}
T_\pi\wt{\mathcal{G}}=\wt{T_x X}\oplus T_\pi({\rm fiber})\cong T_x X\oplus \mathbb{S}'_x,
\end{equation}
where $\wt{T_x X}$ denotes the horizontal lift by $\overline\nabla$, the complex structure 
$\mathbb{I}_r$ can be expressed as
\[
\mathbb{I}_r=I_\pi \oplus I_{\mathbb{S}'_x}.
\]
Here, $I_{\mathbb{S}'_x}$ denotes the standard complex structure on $\mathbb{S}'_x$.


\subsection{The complex structures $\mathbb{J}_r$ on $\wt{\mathcal{G}}$}

As in \S\ref{def-J}, we define an almost complex structure $\mathbb{J}_r$ on $\wt{\mathcal{G}}$ by 
\[
J_r:=J_\pi\oplus \sigma
\]
in terms of the decomposition \eqref{decom-G}, where $J_\pi\colon \mathbb{C}T_xX\to \mathbb{C}T_xX$ is given by
\[
v^{AA'}\frac{\partial}{\partial x^a}\longmapsto v^{AB'}J_{B'}{}^{A'}\frac{\partial}{\partial x^a}
\]
with
\[
J_{B'}{}^{A'}:=\frac{-1}{\|\pi\|_r^2}
\bigl(\overline\e_{C'B'}\pi^{C'}\pi^{A'}+\overline\e_{C'B'}(\sigma\pi)^{C'}(\sigma\pi)^{A'}\bigr).
\]
As in the case of $I_{B'}{}^{A'}$, $J_{B'}{}^{A'}$ is a real operator which is independent of the choice of $r$, and on $\wt{\mathcal{G}}|_{X\setminus \Sigma}$ we may use $\e_{A'B'}, \|\pi\|^2$ in place of $\overline\e_{A'B'}, \|\pi\|_r^2$. Since $J_\pi$ anti-commutes with $I_\pi$ and $\sigma$ is $\mathbb{C}$-anti-linear, we have
\[
\mathbb{I}_r\mathbb{J}_r=-\mathbb{J}_r\mathbb{I}_r.
\]
We can also define $\mathbb{J}_1$ on $\wt{\mathcal{G}}|_{X\setminus\Sigma}$ by using the horizontal lift $\wt{T_x X}$ by $\nabla$ instead of $\overline\nabla$. Then, $\mathbb{J}_1$ is integrable by Theorem \ref{J-integrable}. 

The $(0, 1)$-subspace of $\mathbb{J}_r$ is given by
\begin{align*}
{}^{\mathbb{J}_r}T^{0, 1} _\pi\wt{\mathcal{G}}
&=\wt{\mathcal{A}}_{[\pi-i\sigma(\pi)]}\oplus\{1\otimes\eta+i\otimes\sigma(\eta)\ |\ \eta\in{\mathbb{S}'}_x\} \\
&\subset\mathbb{C}\wt{T_x X}\oplus (\mathbb{C}\otimes_{\mathbb{R}}\mathbb{S}'_x),
\end{align*}
where $\wt{\mathcal{A}}_{[\pi-i\sigma(\pi)]}$ is the horizontal lift by $\overline\nabla$.
Then, by a similar computation to the proof of Proposition \ref{dilation}, we can show that for defining functions $r, \wh r=e^\U r$, the dilations
\[
\d_{e^\U}\colon (\wt{\mathcal{G}}, \mathbb{J}_{\wh r})\longrightarrow (\wt{\mathcal{G}}, \mathbb{J}_{r}), \quad \d_{|r|} \colon (\wt{\mathcal{G}}|_{X\setminus \Sigma}, \mathbb{J}_{r})\longrightarrow (\wt{\mathcal{G}}|_{X\setminus \Sigma}, \mathbb{J}_{1})
\]
are biholomorphisms. In particular, $\mathbb{J}_r$ is also integrable. 


\subsection{Embedding of the twistor CR manifold into the twistor space $\mathbb{P}(\mathbb{S}')$}\label{embedding}

It follows from Proposition \ref{dilation} that the complex structure on $\mathbb{P}(\mathbb{S}')$ induced by $\mathbb{I}_r$ is independent of the choice of $r$, and the hypersurface $\mathbb{P}(\mathbb{S}')|_\Sigma\subset\mathbb{P}(\mathbb{S}')$ is equipped with the induced CR structure. We will prove that there exists a canonical CR isomorphism  $f\colon \mathbb{P}(\mathbb{S}')|_\Sigma\to M$. 
\begin{lem}\label{def-xi}
Let $r$ be a defining function of $\Sigma\subset X$, and let $x\in \Sigma$.Then, for any $\pi^{A'}\in \mathbb{S}'_x\setminus\{0\}$, there exists $\xi^A\in\mathbb{S}_x\setminus\{0\}$, unique up to scalar multiples, satisfying
\begin{equation}\label{tangential}
r_{AA'}\xi^A\pi^{A'}=0.
\end{equation}
\end{lem}
\begin{proof}
It suffices to show that $r_{AA'}\pi^{A'}\neq 0$ as an element of $\mathbb{S}^*_x$. Suppose that
$r_{AA'}\xi^A\pi^{A'}=0$ for any $\xi^A\in\mathbb{S}_x$. Then we have $dr|_{\mathcal{A}_{[\pi]}}=0$.
Since $dr$ is real, it follows that 
\[
dr|_{\mathbb{C}T_xX}=dr|_{{\mathcal{A}_{[\pi]}\oplus{\mathcal{A}_{[\sigma(\pi)]}}}}=0,
\]
which contradicts the fact that $r$ is a defining function.
\end{proof}
Note that $\xi^A$ in Lemma \ref{def-xi} can be taken smooth in $x$ and holomorphic in $\pi^{A'}$. Also, $[\xi]\in\mathbb{P}(\mathbb{S})|_\Sigma$ is independent of the choice of $r$. The equation \eqref{tangential} implies that the complex null vector 
\[
\zeta:=\xi^A\pi^{A'}\frac{\partial}{\partial x^a}
\]
is tangent to $\Sigma$, so we obtain a diffeomorphism
\[
f\colon \mathbb{P}(\mathbb{S}')|_\Sigma\longrightarrow M, \quad [\pi]\longmapsto 
[\zeta]. 
\]
Here we regard $\zeta$ as a complex 1-form on $\Sigma$ by lowering  the index by a metric $h\in[h]$.
We can also (locally) define a map
\[
\wt f\colon \wt{\mathcal{G}}|_\Sigma\longrightarrow \mathbb{C}T\Sigma, \quad \pi\longmapsto \zeta=\xi^A\pi^{A'}\frac{\partial}{\partial x^a},
\]
which descends to $f$ under the identification $\mathbb{C}T\Sigma\cong\mathbb{C}T^*\Sigma$ by a fixed metric $h\in[h]$.
\begin{theorem}\label{CR-isom}
The map $f\colon \mathbb{P}(\mathbb{S}')|_\Sigma\to M$ is a CR isomorphism.
\end{theorem}
\begin{proof}
We take the special defining function $r$ associated with the representative metric $h\in[h]$ used in the definition of $\wt f$. 
Then, since $(\Sigma, h)$ is a totally geodesic submanifold of $(X, \overline{g}_+)$ by Proposition \ref{totally}, we can take local coordinates $(x^a)=(x^0=r, x^i)$ around a point $x_0\in\Sigma$ so that $(x^i)$ give normal coordinates for $h$, and the connection form of the Levi-Civita connection of $\overline g_+$ vanishes at $x_0$.
If $(\zeta^a)=(\zeta^0, \zeta^i)$ are the associated fiber coordinates of $\mathbb{C}TX$, then $(\zeta^i)$ give fiber coordinates of $\mathbb{C}T\Sigma$.

To prove the theorem, it suffices to show that 
\begin{align*}
\wt f_*\bigl(({}^{\mathbb{I}_r}T^{0, 1}\wt{\mathcal{G}})|_\Sigma\cap {\rm Ker}\,dr\bigr)
\subset D+\mathbb{C}\zeta^i\frac{\partial}{\partial \zeta^i}+\mathbb{C}\overline{\zeta^i}\frac{\partial}{\partial \overline{\zeta^i}}=D\oplus\mathbb{C}\zeta^i\frac{\partial}{\partial \zeta^i}
\end{align*}
under the isomorphism $\mathbb{C}T\Sigma\cong\mathbb{C}T^*\Sigma$ by $h$. 
Recall that ${}^{\mathbb{I}_r}T_\pi^{0, 1}\wt{\mathcal{G}}$ is defined by \eqref{T01-I} and $D$ is given by \eqref{D}.

We take  local frames for $\mathbb{S}, \mathbb{S}'$ such that $\Gamma_{cA}{}^{B}(x_0)=\Gamma_{cA'}{}^{B'}(x_0)=0$. Let $\pi\in\wt{\mathcal{G}}_{x_0}$. By the definition of $\xi^A$, it holds that
\[
\wt{\mathcal{A}}_{[\pi]}\cap{\rm Ker}\, dr=\mathbb{C}\xi^A\pi^{A'}\frac{\partial}{\partial x^a}.
\]
Noting that $d\gamma^a_{BB'}(x_0)=d\sigma^{B}{}_{\overline C}(x_0)=d\sigma^{B'}{}_{\!\!\overline {C'}}(x_0)=0$, we have
\begin{align*}
\wt f_*\Bigl(\xi^A\pi^{A'}\frac{\partial}{\partial x^a}\Bigr)
&=\xi^A\pi^{A'}\Bigl(\frac{\partial}{\partial x^a}+\frac{\partial \xi^B}{\partial x^a}\pi^{B'}\frac{\partial}{\partial \zeta^b}+\sigma^{B}{}_{\overline C}\sigma^{B'}{}_{\!\!\overline {C'}}\frac{\partial \overline{\xi^C}}{\partial x^a}\overline{\pi^{C'}}\frac{\partial}{\partial \overline{\zeta^b}}\Bigr) \\
&=\zeta^i\frac{\partial}{\partial x^i}+\zeta^i\frac{\partial \xi^B}{\partial x^i}\pi^{B'}\frac{\partial}{\partial \zeta^b}+\sigma^{B}{}_{\overline C}\sigma^{B'}{}_{\!\!\overline {C'}}\zeta^i\frac{\partial \overline{\xi^C}}{\partial x^i}\overline{\pi^{C'}}\frac{\partial}{\partial \overline{\zeta^b}}.
\end{align*}
Now we claim that 
\begin{equation}\label{proportional}
\zeta^i\frac{\partial \xi^B}{\partial x^i}=\lambda \xi^B, \quad 
\sigma^{B}{}_{\overline C}\zeta^i\frac{\partial \overline{\xi^C}}{\partial x^i}=\lambda' \sigma^B{}_{\overline C}\overline{\xi^C}
\end{equation}
for some constants $\lambda, \lambda'\in\mathbb{C}^*$: If we apply $\zeta^i(\partial/\partial x^i)$ to \eqref{tangential}, we have
\[
\zeta^i\frac{\partial r_{AA'}}{\partial x^i}\xi^A\pi^{A'}+r_{AA'}\zeta^i\frac{\partial \xi^A}{\partial x^i}\pi^{A'}=0.
\]
Since the second fundamental form of $\Sigma$ vanishes, we have 
\[
\zeta^i\frac{\partial r_{AA'}}{\partial x^i}\xi^A\pi^{A'}
=\overline\nabla_i r_j \zeta^i\zeta^j=0
\]
at $x_0$, and hence 
\[
r_{AA'}\zeta^i\frac{\partial \xi^A}{\partial x^i}\pi^{A'}=0.
\]
Thus, by the uniqueness of $\xi^A$ up to scalar multiplications, we 
obtain the first equation of \eqref{proportional}. The second equation can be derived in a similar way by the differentiation of the equation
\[
r_{BB'}\sigma^{B}{}_{\overline C}\sigma^{B'}{}_{\!\!\overline {C'}}\overline{\xi^{C}}\overline{\pi^{C'}}=0,
\]
which is the complex conjugate of \eqref{tangential}. Thus, we have
\[
\wt f_*\Bigl(\xi^A\pi^{A'}\frac{\partial}{\partial x^a}\Bigr)=
\zeta^i\frac{\partial}{\partial x^i}+\lambda\zeta^i\frac{\partial}{\partial \zeta^i}+\lambda'\overline{\zeta^i}\frac{\partial}{\partial \overline{\zeta^i}}\in D\oplus\mathbb{C}\zeta^i\frac{\partial}{\partial \zeta^i}.
\]

Next we consider $\wt f_*\bigl(T^{0, 1}(\rm fiber)\bigr)$. Recalling that $\xi^A$ is holomorphic in $\pi^{A'}$, we have
\begin{align*}
\wt f_*\frac{\partial}{\partial \pi^{A'}}
&=\frac{\partial(\gamma^a_{BB'}\xi^B\pi^{B'})}{\partial \pi^{A'}}\frac{\partial}{\partial \zeta^a} \\
&=\Bigl(\gamma^a_{BB'}\frac{\partial \xi^B}{\partial \pi^{A'}}\pi^{B'}+\gamma^a_{BA'}\xi^B\Bigr)\frac{\partial}{\partial \zeta^a} \\
&=: w^i\frac{\partial}{\partial \zeta^i}.
\end{align*}
Since $\zeta_i w^i=0$, we obtain that 
\[
\wt f_*\frac{\partial}{\partial \overline{\pi^{A'}}}\in D.
\]
Thus we complete the proof.
\end{proof} 

In the proof above, we only use the fact that $(\Sigma, [h])$ is the conformal infinity of 
an anti self-dual Einstein metric, and we do not need its real analyticity. However, it is proved by LeBrun \cite{LeB2} that a twistor CR manifold can be  embeddable to a complex three-manifold if and only if the base conformal manifold is real analytic.
Thus we obtain the following
\begin{theorem}\label{analyticity}
If a $3$-dimensional $C^\infty$ conformal manifold $(\Sigma, [h])$ is the two-sided conformal infinity of an anti self-dual Einstein metric, then it is real analytic.
\end{theorem}


\subsection{The hyperk\"ahler ambient metrics on $\wt{\mathcal{G}}$}
We consider the holomorphic 1-form 
\[
\wt\tau_1:=\pi_{B'}\d\pi^{B'}=\e_{A'B'}\pi^{A'}\d\pi^{B'}
\]
on $(\wt{\mathcal{G}}|_{X\setminus\Sigma}, \mathbb{I}_{1})$ as in \S\ref{hyper}. 
By pulling back $\wt\tau_1$ via the dilation $\d_{|r|}\colon \wt{\mathcal{G}}|_{X\setminus\Sigma}\to \wt{\mathcal{G}}|_{X\setminus\Sigma}$ and multiplying it by the signature of $r$, we define
\[
\wt\tau_r:=\frac{r}{|r|}\d_{|r|}^*\wt\tau_1=r\overline{\e}_{A'B'}\pi^{A'}\d\pi^{B'},
\]
which is  holomorphic with respect to $\mathbb{I}_r$. By the conformal transformation formula
\[
\overline{\Gamma}_{eC'}{}^{B'}={\Gamma}_{eC'}{}^{B'}+\d_{E'}{}^{B'}r^{-1}r_{EC'},
\]
which follows from \eqref{spinor-conf}, we have
\begin{equation}\label{tau}
\wt\tau_r=r\overline{\e}_{A'B'}\pi^{A'}(d\pi^{B'}+\overline{\Gamma}_{eC'}{}^{B'}\pi^{C'}\theta^e)
-\overline{\e}_{A'E'}r_{EC'}\pi^{A'}\pi^{C'}\theta^e.
\end{equation}
Thus, $\wt\tau_r$ extends to a holomorphic 1-form on $(\wt{\mathcal{G}}, \mathbb{I}_r)$.
We also define 
\[
\wt\omega_r:=d\wt\tau_r
\]
and 
\begin{equation}\label{omega-J-r}
\wt\omega_{\mathbb{J}_r}(V, W):=\frac{-i}{2}\bigl(\wt\omega_r(V, \mathbb{J}_r W)-\wt\omega_r(W, \mathbb{J}_r V)\bigr)
\end{equation}
on $\wt{\mathcal{G}}$. Similarly, we have $\wt\omega_1$ and $\wt\omega_{\mathbb{J}_1}$ on $\wt{\mathcal{G}}|_{X\setminus\Sigma}$. Note that for defining functions $r, \wh r=e^\U r$, these forms are related to each other by
\[
\d_{e^\U}^*\wt\omega_{\mathbb{J}_r}=\wt\omega_{\mathbb{J}_{\wh r}}, \quad 
\frac{r}{|r|}\d_{|r|}^*\wt\omega_{\mathbb{J}_1}=\wt\omega_{\mathbb{J}_r}.
\]
We define a function $\wt r$ on $\wt{\mathcal{G}}$ by
\[
\wt r:=\|\pi\|^2_r\, r,
\]
where $\|\pi\|^2_r$ is given by \eqref{norm-r}.
Then, $\wt r$ is a defining function of $\wt{\mathcal{G}}|_{\Sigma}\subset\wt{\mathcal{G}}$ which is homogeneous of degree $(1, 1)$. On $\wt{\mathcal{G}}|_{X\setminus\Sigma}$, we have
\[
\frac{r}{|r|}\d_{|r|}^* (\|\cdot\|^2) =\wt r,
\]
where 
$\|\pi\|^2=\sigma^{C'}{}_{\!\!\overline{B'}}{\e}_{C'A'}\pi^{A'}\overline{\pi^{B'}}=\sigma_{A'\overline{B'}}\pi^{A'}\overline{\pi^{B'}}$. Since $g_+$ is anti self-dual Einstein, we have
\[
i\partial\overline{\partial}\|\pi\|^2=\wt\omega_{\mathbb{J}_1}
\]
on $(\wt{\mathcal{G}}|_{X\setminus\Sigma}, \mathbb{I}_1)$ by Theorem \ref{thm-potential}. Pulling back this equation by  the dilation
\[
\d_{|r|}\colon (\wt{\mathcal{G}}|_{X\setminus\Sigma}, \mathbb{I}_r)\longrightarrow (\wt{\mathcal{G}}|_{X\setminus\Sigma}, \mathbb{I}_1)
\]
and multiplying it by $r/|r|$, we obtain
\begin{equation*}\label{potential-r}
i\partial\overline{\partial}\,\wt r=\wt\omega_{\mathbb{J}_r}
\end{equation*}
on  $(\wt{\mathcal{G}}|_{X\setminus\Sigma}, \mathbb{I}_r)$ and hence on $(\wt{\mathcal{G}}, \mathbb{I}_r)$.
As the Levi form of $M$ has the Lorentzian signature, $i\partial\overline{\partial}\,\wt r$ is the K\"ahler form of a neutral K\"ahler metric $\wt g[r]$ on $\wt{\mathcal{G}}$ (in a neighborhood of $\wt{\mathcal{G}}|_{\Sigma}$). Moreover, since $\wt\omega_{\mathbb{J}_1}$ is hyperk\"ahler by Theorem \ref{J-parallel}, so is $\wt g[r]$. In particular, it is Ricci-flat and gives an ambient metric associated with the twistor CR manifold $M\cong\mathbb{P}(\mathbb{S}')|_\Sigma$. Thus we obtain Theorem \ref{main-theorem}.

We note that by Proposition \ref{canonical-bundle}, the $3$-form $\wt\tau_r\wedge d\wt\tau_r$ defines an isomorphism 
\begin{equation*}
\wt{\mathcal{G}}\cong K_{\mathbb{P}(\mathbb{S}')}^{\frac{1}{4}}\setminus\{\bo\}
\end{equation*}
as $\mathbb{C}^*$-bundles over $\mathbb{P}(\mathbb{S}')$.


\subsection{The smooth Cheng--Yau metric on $\mathbb{P}(\mathbb{S}')|_{X\setminus\Sigma}$}

Since the K\"ahler potential $\wt r$ solves the complex Monge--Amp\`ere equation, we 
have the smooth Cheng--Yau metric $g_{\rm CY}$ with the K\"ahler form
\[
\omega_{\rm CY}:=-i\partial\overline\partial\log|\wt r|,
\]
which is a Lorentzian K\"ahler-Einstein metric on 
$\mathbb{P}(\mathbb{S}')|_{X\setminus\Sigma}$.
Note that since the complex manifolds $(\wt{\mathcal{G}}, \mathbb{I}_r)$ are biholomorphic to each other by dilations, the induced complex structures on $\mathbb{P}(\mathbb{S}')$ are all the same, and on $\mathbb{P}(\mathbb{S}')|_{X\setminus\Sigma}$  the complex structure is also induced by $\mathbb{I}_1$. Moreover, the K\"ahler potentials $\wt r$ are related by dilations, so the smooth Cheng--Yau metric 
does not depend on the choice of a defining function $r$, and we can also use $\|\pi\|^2$ in place of $|\wt r|$:
\[
\omega_{\rm CY}=-i\partial\overline\partial\log\|\pi\|^2.
\] 
Thus, from Theorem \ref{Kahler-Einstein}, we obtain the following description of the smooth Cheng--Yau metric:  
\begin{theorem}\label{cheng-yau}
Let $H\subset T\mathbb{P}(\mathbb{S}')|_{X\setminus\Sigma}$ be the distribution obtained by the horizontal lift of $T(X\setminus\Sigma)$ with respect to the connection $\nabla$ determined by $g_+$. Then, in terms of the decomposition
\[
T\mathbb{P}(\mathbb{S}')|_{X\setminus\Sigma}=H\oplus T({\rm fiber}),
\]
the smooth Cheng--Yau metric is given by
\[
g_{\rm CY}=(-\Lambda g_+)\oplus(-g_{\rm FS}),
\]
where $g_{\rm FS}$ denotes the Fubini--Study metric on each fiber $\mathbb{P}(\mathbb{S}'_x)\cong \mathbb{P}^1$.
\end{theorem}


\section{The flat case}\label{flat-case}
\subsection{Description of the flat twistor CR manifold}
We consider the case when $(\Sigma, [h])$ is the flat conformal three-manifold 
$(\mathbb{R}^3, [(dx^1)^2+(dx^2)^2+(dx^3)^2])$, which is the conformal infinity of the hyperbolic metric
\[
g_+=\frac{(dx^0)^2+(dx^1)^2+(dx^2)^2+(dx^3)^2}{(x^0)^2}
\]
with $\Lambda=-1/2$ on $X\setminus\Sigma\ (X:=\mathbb{R}^4)$. We fix the defining function $r=x^0$ and the compactified metric 
$\overline g_+=(dx^0)^2+(dx^1)^2+(dx^2)^2+(dx^3)^2$. If we identify a point $(x^a)\in X$ with the matrix
\[
A(x)=\frac{1}{\sqrt{2}}
\begin{pmatrix}
x^0+ix^3 & x^1+ix^2 \\
-x^1+ix^2 & x^0-ix^3
\end{pmatrix}=:(x^{AA'}),
\]
then we have 
\[
\overline g_+=2\det A(dx)=\overline\e_{AB}\overline\e_{A'B'}dx^{AA'}dx^{BB'}
\]
with 
\[
(\overline\e_{AB})=(\overline\e_{A'B'})=
\begin{pmatrix}
0 & 1 \\
-1 & 0
\end{pmatrix}.
\]
Thus, we have spinor bundles $\mathbb{S}=\mathbb{S}'=X\times\mathbb{C}^2$ with the isomorphism $\mathbb{C}TX\cong\mathbb{S}\otimes \mathbb{S'}$ given by $v^a\leftrightarrow v^{AA'}=A(v^a)$. In the sequel, we raise and lower spinor indices by $\overline\e_{AB}, \overline\e_{A'B'}$ and their inverses. The $\mathbb{C}$-anti-linear bundle map $\sigma\colon 
\mathbb{S}'\to\mathbb{S}'$ is given by
\[
\sigma\begin{pmatrix}
\pi^{0'} \\
\pi^{1'}
\end{pmatrix}
=\begin{pmatrix}
\overline{\pi^{1'}} \\
-\overline{\pi^{0'}}
\end{pmatrix}
=\begin{pmatrix}
0 & 1 \\
-1 & 0
\end{pmatrix}
\begin{pmatrix}
\overline{\pi^{0'}} \\
\overline{\pi^{1'}}
\end{pmatrix}.
\]

The twistor space for the flat metric is well-known in twistor theory (e.g., \cite{PR1}), and  
the complex manifold $(\wt{\mathcal{G}}=\mathbb{S}'\setminus\{\bo\}, \mathbb{I}_r)$ can be described as follows. We define a differmorphism
\[
\wt F\colon \wt{\mathcal{G}}=X\times (\mathbb{C}^2\setminus\{0\})\longrightarrow 
\mathbb{C}^4\setminus\mathbb{C}^2
=\Bigl\{ \Bigl(\begin{pmatrix}
W^0 \\
W^1
\end{pmatrix}, 
\begin{pmatrix}
W^2 \\
W^3
\end{pmatrix}
\Bigr)\ \Big|\   \begin{pmatrix}
W^2 \\
W^3
\end{pmatrix}\neq \begin{pmatrix}
0 \\
0
\end{pmatrix}\Bigr\}
\]
by 
\[
\bigl(x^a, \pi^{A'}\bigr)\longmapsto \bigl(  x^{AA'}\pi_{A'}, \pi_{A'} \bigr)
=\Bigl(A(x)\begin{pmatrix}
-\pi^{1'} \\
\pi^{0'}
\end{pmatrix}, \begin{pmatrix}
-\pi^{1'} \\
\pi^{0'}
\end{pmatrix}\Bigr).
\]
Note that $A(x)$ satisfies 
\begin{equation}\label{A-property}
A(x)\begin{pmatrix}
0 & 1 \\
-1 & 0
\end{pmatrix}
=\begin{pmatrix}
0 & 1 \\
-1 & 0
\end{pmatrix}\overline{A(x)},
\end{equation}
and hence when $\wt F(x^a, \pi^{A'})=(W^\a)$, we have
\[
A(x)\begin{pmatrix}
W^2  & -\overline{W^3} \\
W^3 & \overline{W^2}
\end{pmatrix}
=\begin{pmatrix}
W^0  & -\overline{W^1} \\
W^1 & \overline{W^0}
\end{pmatrix}, 
\]
from which one can construct the inverse of $\wt F$.
\begin{prop}
The map $\wt F\colon (\wt{\mathcal{G}},  \mathbb{I}_r)\to \mathbb{C}^4\setminus\mathbb{C}^2$ is a biholomorphism.
\end{prop}
\begin{proof}
We will show that $\wt F_*(T^{0, 1}\wt{\mathcal{G}})\subset T^{0, 1}(\mathbb{C}^4\setminus\mathbb{C}^2)$. Since the restriction of $\wt F$ to each fiber is holomorphic, it suffices to consider $\wt F_*(\wt{\mathcal{A}}_{[\pi]})$. Any $v\in \wt{\mathcal{A}}_{[\pi]}$ can be written as $v=\xi^A\pi^{A'}(\partial/\partial x^a)$ since $\overline g_+$ is a flat metric. Thus, we have $dW^2(\wt F_*v)=dW^3(\wt F_*v)=0$ and 
\[
\begin{pmatrix}
dW^0(\wt F_* v) \\
dW^1(\wt F_* v)
\end{pmatrix}
=\bigl(v^{AA'}\pi_{A'}\bigr)=0,
\] 
which imply $\wt F_*v\in T^{0, 1}(\mathbb{C}^4\setminus\mathbb{C}^2)$.
\end{proof}
Since $\wt F$ is $\mathbb{C}^*$-equivariant, it descends to a biholomorphism
\[
F\colon \mathbb{P}(\mathbb{S}')\longrightarrow \mathbb{P}^3\setminus\mathbb{P}^1.
\]
The submanifold $\Sigma\subset X$ is characterized by the condition ${\rm tr} A(x)=0$, so $F$ maps the twistor CR manifold $M\cong \mathbb{P}(\mathbb{S}')|_{\Sigma}$ to the real hypersurface
\[
\mathcal{N}=\bigl\{ [(W^\a)]\in \mathbb{P}^3\setminus\mathbb{P}^1\ \big|\ 
W^0\overline{W^2}+W^2\overline{W^0}+W^1\overline{W^3}+W^3\overline{W^1}=0\bigr\}.
\]

\subsection{Computation of $\mathbb{J}_r$}
The second complex structure $\mathbb{J}_r$ on $\wt{\mathcal{G}}$ defines a complex structure on $\mathbb{C}^4\setminus\mathbb{C}^2$ via $\wt F$. We will denote it by the same symbol $\mathbb{J}_r$ and compute it in terms of the coordinates $(W^\a)$. To simplify the notation, let us denote a real tangent vector $w^\a(\partial/\partial W^\a)+({\rm c.c.})$ on $\mathbb{C}^4\setminus\mathbb{C}^2$ by
\[
(w^0, w^1, w^2, w^3)+({\rm c.c.}),
\] 
where $({\rm c.c.})$ stands for the complex conjugate. 
\begin{prop}\label{J-flat} The complex structure $\mathbb{J}_r$ on $\mathbb{C}^4\setminus\mathbb{C}^2$ is given by
\[
\mathbb{J}_r\bigl((w^0, w^1, w^2, w^3)+({\rm c.c.})\bigr)
=(\overline{w^1}, -\overline{w^0}, \overline{w^3}, -\overline{w^2})+({\rm c.c.}).
\]
\end{prop}
\begin{proof}
Since $\overline g_+$ is flat, the complex structure $\mathbb{J}_r$ on $\wt{\mathcal{G}}$ is given by $J_\pi\oplus\sigma$ in terms of the decomposition $T_\pi \wt{\mathcal{G}}=TX\oplus \mathbb{C}^2$. Thus, for a vertical tangent vector 
\[
(w^0, w^1, w^2, w^3)+({\rm c.c.})=\wt F_*\Bigl(\lambda^{A'}\Bigl(\frac{\partial}{\partial \pi^{A'}}\Bigr)_{(x, \pi)}+({\rm c.c.})\Bigr),
\]
we have
\begin{align*}
\mathbb{J}_r\bigl((w^0, w^1, w^2, w^3)+({\rm c.c.})\bigr)
&=\Bigl(A(x)\begin{pmatrix}
\overline{\lambda^{1'}} \\
-\overline{\lambda^{0'}}
\end{pmatrix}, 
\begin{pmatrix}
\overline{\lambda^{1'}} \\
-\overline{\lambda^{0'}}
\end{pmatrix}
\Bigr)+({\rm c.c.}) \\
&=(\overline{w^1}, -\overline{w^0}, \overline{w^3}, -\overline{w^2})+({\rm c.c.})
\end{align*}
by \eqref{A-property}. For a horizontal tangent vector 
\[
(w^0, w^1, 0, 0)+({\rm c.c.})=\wt F_*\Bigl(v^a\Bigl(\frac{\partial}{\partial x^a}\Bigr)_{(x, \pi)}\Bigr)=\bigl(v^{AA'}\pi_{A'}, 0, 0\bigr)+({\rm c.c.}),
\]
we have 
\begin{align*}
\mathbb{J}_r\bigl((w^0, w^1, 0, 0)+({\rm c.c.})\bigr)
&=\bigl(v^{AC'}J_{C'}{}^{A'}\pi_{A'}, 0, 0\bigr)+({\rm c.c.}) \\
&=\bigl(v^{AC'}(\sigma\pi)_{C'}, 0, 0\bigr)+({\rm c.c.}) \\
&=(\overline{w^1}, -\overline{w^0}, 0, 0)+({\rm c.c.}),
\end{align*}
where we again used \eqref{A-property} with $x$ replaced by $v$. Thus we complete the proof.
\end{proof}


\subsection{Computation of $\wt g[r]$}
First, we will compute the holomorphic 1-form
\[
\wt\tau_r=r\overline{\e}_{A'B'}\pi^{A'}\d\pi^{B'}
\]
on $\wt{\mathcal{G}}\cong \mathbb{C}^4\setminus\mathbb{C}^2$ for $r=x^0$. By \eqref{tau} and 
$\overline{\Gamma}_{eC'}{}^{B'}=0$, we have
\begin{align*}
\wt\tau_r&=r\overline{\e}_{A'B'}\pi^{A'}d\pi^{B'}
-\overline{\e}_{A'E'}r_{EC'}\pi^{A'}\pi^{C'}dx^e \\
&=r\overline{\e}^{A'B'}\pi_{A'}d\pi_{B'}
-r_{EC'}\pi_{E'}\pi^{C'}dx^{EE'}.
\end{align*}
Since $r=x^0=(1/\sqrt{2})(x^{00'}+x^{11'})$, the components $r_{AA'}$ are given by
\[
r_{00'}=r_{11'}=\frac{1}{\sqrt{2}}, \quad r_{01'}=r_{10'}=0.
\]
Thus, we have
\begin{align*}
&\quad \sqrt{2}\bigl(r\overline{\e}^{A'B'}\pi_{A'}d\pi_{B'}-r_{EC'}\pi_{E'}\pi^{C'}dx^{EE'}\bigr) \\
&=
(x^{00'}+x^{11'})(\pi_{0'}d\pi_{1'}-\pi_{1'}d\pi_{0'}) \\
&\quad 
-(\pi^{0'}\pi_{0'}dx^{00'}+\pi^{0'}\pi_{1'}dx^{01'}
+\pi^{1'}\pi_{0'}dx^{10'}+\pi^{1'}\pi_{1'}dx^{11'}) \\
&=(x^{00'}\pi_{0'}+x^{01'}\pi_{1'})d\pi_{1'}-\pi_{1'}d(x^{00'}\pi_{0'}+x^{01'}\pi_{1'})\\
&\quad +\pi_{0'}d(x^{11'}\pi_{1'}+x^{10'}\pi_{0'})-(x^{11'}\pi_{1'}+x^{10'}\pi_{0'})d\pi_{0'}\\
&=W^0dW^3-W^3dW^0+W^2dW^1-W^1dW^2
\end{align*}
and hence
\begin{equation*}
\begin{aligned}
\wt\tau_r&=\frac{1}{\sqrt{2}}(W^0dW^3-W^3dW^0+W^2dW^1-W^1dW^2), \\
\wt\omega_r&=\sqrt{2}(dW^0\wedge dW^3-dW^1\wedge dW^2).
\end{aligned}
\end{equation*}
It follows from this formula of $\wt \omega_r$ and Proposition \ref{J-flat} that the 
K\"ahler form of the ambient metric $\wt g[r]$ is given by 
\begin{equation}\label{omega-J-flat}
\begin{aligned}
\wt\omega_{\mathbb{J}_r}
&=\frac{-i}{2}\cdot\sqrt{2}\bigl((dW^0\circ\mathbb{J}_r)\wedge dW^3+dW^0\wedge
(dW^3\circ\mathbb{J}_r) \\
&\qquad\qquad\qquad -(dW^1\circ\mathbb{J}_r)\wedge dW^2-dW^1\wedge (dW^2\circ\mathbb{J}_r)\bigr) \\
&=\frac{i}{\sqrt{2}}\bigl(dW^0\wedge d\overline{W^2}+
dW^2\wedge d\overline{W^0}+dW^1\wedge d\overline{W^3}+dW^3\wedge d\overline{W^1}\bigr).
\end{aligned}
\end{equation}
Thus, we obtain the following 
\begin{theorem}
The hyperk\"ahler ambient metric $\wt g[r]$ on $\wt{\mathcal{G}}\cong \mathbb{C}^4\setminus\mathbb{C}^2$ associated with the twistor CR manifold $\mathcal{N}\subset\mathbb{P}^3\setminus\mathbb{P}^1$ is given by the flat metric
\[
\wt g[r]=\sqrt{2}(dW^0\cdot d\overline{W^2}+dW^2\cdot d\overline{W^0}+dW^1\cdot d\overline{W^3}+dW^3\cdot d\overline{W^1}).
\]
\end{theorem}

\subsection{Computation of $\wt r$}
By \eqref{potential-r}, the K\"ahler potential of the ambient metric $\wt g[r]$ is given by
\[
\wt r=\|\pi\|^2_r r=\frac{1}{\sqrt{2}}(|\pi_{0'}|^2+|\pi_{1'}|^2)(x^{00'}+x^{11'}).
\]
On the other hand, using $\overline{x^{00'}}=x^{11'}$ and $\overline{x^{01'}}=-x^{10'}$, we have
\begin{align*}
2{\rm Re}\,(W^0\overline{W^2}+W^1\overline{W^3}) 
&=2{\rm Re}\,\bigl((x^{00'}\pi_{0'}+x^{01'}\pi_{1'})\overline{\pi_{0'}}+(x^{10'}\pi_{0'}+x^{11'}\pi_{1'})\overline{\pi_{1'}}\bigr) \\
&=(|\pi_{0'}|^2+|\pi_{1'}|^2)(x^{00'}+x^{11'}).
\end{align*}
Thus, we obtain 
\[
\wt r=\frac{1}{\sqrt{2}}(W^0\overline{W^2}+W^2\overline{W^0}+W^1\overline{W^3}+W^3\overline{W^1}), 
\]
from which we recover the formula \eqref{omega-J-flat}.

\end{document}